\documentclass[reqno]{amsart}
\usepackage{amsmath, amsthm, amssymb, mathrsfs, thmtools, enumerate}
\usepackage[colorlinks = true, linkcolor = black, urlcolor = blue, citecolor = black, pdfborder={0 0 0}]{hyperref} 
\usepackage{tikz} 

\allowdisplaybreaks[4]
\declaretheoremstyle[
  spaceabove=1em, spacebelow=1em,
  headfont= \bfseries,
  notefont=\mdseries, notebraces={(}{)},
  bodyfont=\normalfont,
  postheadspace=1em,
  qed= \tiny $\blacksquare$
]{example}

\declaretheoremstyle[
  spaceabove=\topsep, spacebelow=\topsep,
  headfont= \bfseries,
  notefont=\mdseries, notebraces={(}{)},
  bodyfont=\normalfont,
  postheadspace=1em,
  qed= $\blacktriangledown$
]{remark}

\declaretheoremstyle[
  spaceabove=\topsep, spacebelow=\topsep,
  headfont= \bfseries,
  notefont=\mdseries, notebraces={(}{)},
  bodyfont=\it{\normalfont},
  postheadspace=1em,
  qed= \qedsymbol
]{obvious}

\theoremstyle{plain}
\newtheorem{Thm}{Theorem}[section]
\newtheorem{Prop}[Thm]{Proposition}
\newtheorem{Cor}[Thm]{Corollary}
\newtheorem{Lem}[Thm]{Lemma}

\theoremstyle{definition}
\newtheorem{Def}[Thm]{Definition}

\numberwithin{equation}{section}
\numberwithin{table}{section}

\newcommand{\bmid}{\;\big|\;}
\newcommand{\C}{\mathbb{C}}
\newcommand{\Z}{\mathbb{Z}}
\newcommand{\Q}{\mathbb{Q}}
\newcommand{\QG}{\Q[G]}
\newcommand{\QH}{\Q[H]}
\newcommand{\FG}{F[G]}
\newcommand{\rr}{\mathcal{R}}
\newcommand{\s}{\mathcal{S}}
\newcommand{\G}{\mathcal{G}}
\renewcommand{\H}{\mathcal{H}}
\renewcommand{\L}{\mathcal{L}}
\newcommand{\K}{\mathcal{K}}
\newcommand{\D}{\mathcal{D}}
\newcommand{\dsum}{\displaystyle\sum}

\DeclareMathOperator{\Aut}{Aut}
\DeclareMathOperator*{\Inn}{Inn}
\DeclareMathOperator{\Span}{Span}

\newcommand{\thmref}[1]{Theorem \ref{#1}}
\newcommand{\corref}[1]{Corollary \ref{#1}}
\newcommand{\lemref}[1]{Lemma \ref{#1}}
\newcommand{\propref}[1]{Proposition \ref{#1}}
\newcommand{\defref}[1]{Definition \ref{#1}}
\newcommand{\tableref}[1]{Table \ref{#1}}
\newcommand{\secref}[1]{Section \ref{#1}}
\newcommand{\figref}[1]{Figure \ref{#1}}

\newcommand{\bigwr}{\scalebox{2}{$\wr$}}
\renewcommand{\O}{\mathcal{O}}

\newcommand{\specialcell}[2][c]{\begin{tabular}[#1]{@{}c@{}}#2\end{tabular}}

\begin{document}

\title{Counting Schur Rings over Cyclic Groups}
\author{Andrew Misseldine}

\begin{abstract}
{Any Schur ring is uniquely determined by a partition of the elements of the group. An open question in the study of Schur rings is determining which partitions of the group induce a Schur ring. Although a structure theorem is available for Schur rings over cyclic groups, it is still a difficult problem to count all the partitions. For example, Kovacs, Liskovets, and Poschel  determine formulas to count the number of wreath-indecomposable Schur rings. In this paper we solve the problem of counting the number of all Schur rings over cyclic groups of prime power order and draw some parallels with Higman's PORC conjecture.
}



\end{abstract}

\maketitle

\section{Introduction}
Let $F$ be a field of characteristic zero, and let $A$ denote an $F$-algebra. For any finite subset $C\subseteq A$, let $\overline{C} = \sum_{x\in C} x\in A$.  Let $G$ denote a finite group, and let $\FG$ denote the group algebra of $G$ with coefficients from $F$. We say that an element $\alpha\in \FG$ is a \emph{simple quantity} if there exists some subset $C\subseteq G$ such that $\alpha = \overline{C}$. Let $\{C_1, C_2, \ldots, C_r\}$ be a partition of a finite group $G$, and let $S$ be the subspace of $\FG$ spanned by $\overline{C_1}, \overline{C_2},\ldots \overline{C_r}$, that is, $S$ is spanned by simple quantities and contains the element $\overline{G}$. We say that $S$ is a \emph{Schur ring} \cite[Wielandt]{Wielandt49} over $G$ if 
\begin{enumerate}
\item $C_1 = \{1\}$, 
\item For each $i$, there is a $j$ such that $(C_i)^{-1} = C_j$,
\item For each $i$ and $j$, $\overline{C_i}\cdot\overline{C_j} = \dsum_{k=1}^r \lambda_{ijk}\overline{C_k}$, for $\lambda_{ijk}\in F$.
\end{enumerate} Schur rings were originally developed by Schur and Wielandt in the first half of the 20th century. Schur rings were first used to study permutation groups, but in later decades applications of Schur rings have emerged in combinatorics, graph theory, and design theory \cite{KlinPoschel, Ma}.

As mentioned above, any Schur ring is uniquely determined by a partition of the elements of the group, although not every partition determines a Schur ring. An open question in the study of Schur rings is determining which partitions of the group induce a Schur ring and which ones do not. Much work has be done to answer this question. In the case that our group $G$ is cyclic, a complete classification has been found \cite{LeungII, LeungI}; see \thmref{thm:LeungMan}. In particular, the study of Schur rings over cyclic groups is a very active field with several recent papers being published on this topic: \cite{Kovacs}, \cite{Leung90}, \cite{LeungII}, \cite{LeungI}, \cite{Muzychuk98}, \cite{Muzychuk93}, and \cite{Muzychuk94}.

In this paper we consider the problem of counting the number of Schur rings over $Z_n$, the cyclic group of order $n$. Although a structure theorem is available for Schur rings over cyclic groups, this still proves to be a difficult problem. Specializations of this problem have been considered before. For example, in \cite{Kovacs} Kov\'acs determines a formula to count the number of Schur rings over $Z_{2^n}$ which are \emph{wreath-indecomposable}, that is, those Schur rings which cannot be properly factored as a wreath product of Schur rings (see page \pageref{wreath}). Kov\'acs' formula involves the Catalan and Schr\"oder numbers. We will see these again when we consider Schur rings over cyclic 2-groups in \secref{sec:CountEven}. In \cite{Liskovets}, Liskovets and P\"oschel determine a formula for wreath-indecomposable Schur rings over $Z_{p^n}$, where $p$ is an odd prime. This formula depends on the Catalan numbers and the number of divisors of $p-1$. We will also see these quantities again when we consider Schur rings over cyclic $p$-groups in \secref{chap:Counting}.

A function $f$  on a subset of integers $S$ is called \emph{polynomial on residue classes} or \emph{PORC} if there exists a positive integer $m$ and rational polynomials $p_0,\ p_1,\ldots, p_{m-1}$ such that, for all $x\in S$, $f(x) = p_a(x)$ whenever $x\equiv a\pmod m$ and $0\le a< m$. In other words, $f$ acts like a polynomial on the residue classes modulo $m$. For a natural number $n$ and prime number $p$, let $G(p,n)$ denote the number of isomorphism classes of groups of order $p^n$. The Higman's PORC conjecture states that for a fixed $n$ and $S$ is the set of primes $G(p,n)$ is a PORC function. Because a function on primes is PORC if and only if it is PORC on all but finitely many primes, it suffices to show that a function is PORC for sufficiently large primes. The conjecture has been varied for $n \le 7$ and we include in \tableref{tab:porc} these polynomials (see \cite{VaughanLeeI, VaughanLeeII} and their references for the details).

\begin{table}[hbtp]
\caption{$G(p,n)$ for $n\le 7$}
\begin{center}
\begin{tabular}{|c|c|}
\hline
$n$ & $G(p,n)$\\\hline
1 & $1$ \\\hline
2 & $2$ \\\hline
3 & $5$ \\\hline
4 & $15$ if $p\ge 3$ \\\hline
5 & $2p + 61 + 2\gcd(p-1,3) + \gcd(p-1,4)$ if $p\ge 5$\\\hline
6 & $3p^2+39p +344 +24\gcd(p-1,3) + 11\gcd(p-1,4) + 2\gcd(p-1,5)$ if $p\ge 5$\\\hline
7 & \specialcell{$3p^5+12p^4+44p^3+170p^2+707p+2455$\\ $+\ (4p^2+44p+291)\gcd(p-1,3)+ (p^2+19p+135)\gcd(p-1,4) $\\ $+\ (3p+31)\gcd(p-1,5) + 4\gcd(p-1,7) + 5\gcd(p-1,8) + \gcd(p-1,9)$ if $p\ge 7$} \\\hline
\end{tabular}
\end{center}
\label{tab:porc}
\end{table}

Although Higman's PORC conjecture is likely untrue for $n\ge 10$ \cite{VaughanLeeIII}, various PORC conjectures have been proven for specific families of $p$-groups, e.g. \cite{HigmanPORCI, HigmanPORCII, evseev2008higman, WittyPORC}. In the language of the PORC conjecture, \thmref{thm:formula3rd} shows that the number of Schur rings over a cyclic $p$-group is satisfies the PORC conjecture for all $n$. We mention that the polynomials in \tableref{tab:porc} depend on powers of $p$ and $\gcd(p-1,k)$ for various integers $k$. On the other hand, the polynomials for the number of Schur rings depends entirely only $d(p-1)$, the number of divisors of $p-1$ which we will simply abbreviate as $x$.

\section{Orbit Algebras and Cyclotomic Fields}\label{sec:cycFields}

 Let $A$ be an algebra over a field $F$ and let $\H \le \Aut_F(A)$ be finite, where $\Aut_F(A)$ is the group of $F$-algebra automorphisms of $A$. Let 
\[A^\H = \{\alpha \in A : \sigma(\alpha) = \alpha,\;\text{for all}\;\sigma\in \H\}.\] Then $A^\H$ is a subalgebra of $A$ and is referred to as an \emph{orbit algebra}. The orbit algebra $A^\H$ is, in fact, the largest subalgebra of $A$ that is fixed by all elements of $\H$. Such subalgebras appear frequently in mathematics, especially in Galois theory. 

Suppose $\mathcal{B} \subseteq A$ such that $A = \Span_F(\mathcal{B})$. For each $\alpha\in A$, let $\O_\alpha = \{\sigma(\alpha) : \sigma \in \H\}$ denote the orbit of $\alpha$ in $A$ with respect to $\H$ and let
\[\overline{\O_\alpha} = \sum_{\beta \in \O_\alpha} \beta\] denote the \emph{period} of $\alpha$ with respect to $\H$. Then $A^\H = \Span_F\{\overline{\O_\alpha} : \alpha\in \mathcal{B}\}$, that is, $A^\H$ is spanned by the periods of a spanning set of $A$. This fact is the reason orbit algebras have their name.

One example of orbit algebras that will be of use in this paper will be the subfields of a cyclotomic field. Let $\zeta_n = e^{2\pi i/n}\in \C$, a primitive $n$th root of unity, and let $\K_n = \Q(\zeta_n)\subseteq \C$, the corresponding cyclotomic field. Let $\G_n$ denote the Galois group of $\K_n$ over $\Q$. When the context is clear, the subscripts may be omitted. From Galois theory, we know there is a one-to-one correspondence between the subfields of $\K_n$ and the subgroups of $\G_n$. Since every subalgebra of a field is likewise a field, there is a natural correspondence between the $\Q$-subalgebras of $\K_n$ and the subgroups of $\G_n$. In particular,  if $\H \le \G_n$, then $\K_n^\H = \Q(\overline{\O_{\zeta^i_n}} : 0\le i < n)$. By Galois correspondence, every subfield of $\K_n$ is of this form.

Every automorphism on $\K_n$ is determined by the image of $\zeta_n$, which must be a primitive $n$th root. Therefore, $\G_n\cong (\Z/n\Z)^\ast$, where $(\Z/n\Z)^\ast$ denotes the multiplicative subgroup of the ring $\Z/n\Z$ consisting of congruence classes relatively prime to $n$. For each integer $m$ relatively prime to $n$, let $\sigma_m$ denote the  automorphism in $\G_n$ which is determined by $m$. The group $(\Z/n\Z)^\ast$ is, of course, well-known.

\begin{Lem}\label{prop:AutGroups} \mbox{}
\begin{enumerate}[(a)]
\item $(\Z/{2^k}\Z)^\ast \cong Z_2 \times Z_{2^{k-2}}$, for all $k\ge 2$. In the case that $k =1$, $(\Z/2\Z)^\ast = 1$.
\item $(\Z/{p^k}\Z)^\ast \cong Z_{p^{k-1}(p-1)}$, for $k\ge 1$ and $p$ is an odd prime. 
\item Let $n\ge 2$ be an integer with prime factorization $n = p_1^{k_1}\cdots p_r^{k_r}$ and each $p_i$ is a distinct prime. Then $(\Z/n\Z)^\ast \cong (\Z/{p_1^{k_1}}\Z)^\ast\times \ldots \times (\Z/{p_r^{k_r}}\Z)^\ast$.
\end{enumerate}
\end{Lem}


For each divisor $d$ of $n$, there is a natural quotient map $\G_n \to \G_d$ given by restriction, that is, each automorphism $\sigma : \K_n \to \K_n$ maps to its restriction $\sigma|_{\K_d} : \K_d \to \K_d$. Thus, each subgroup $\H\le \G_n$ induces a unique subgroup of $\G_d$. By abuse of notation, we will denote this quotient group also as $\H$. Since $\H$ can be identified as a set of integers modulo $n$, we may also identify $\H$ with this same set of integers but instead modulo $d$. Then it follows from above that \begin{equation}\label{cor:subfieldPeriod} \K_n^\H\cap \K_d = \K_d^{\H}.\end{equation} Furthermore, $\K_d^{\H}$ is the maximal subfield of $\K_d$ contained in $\K_n^\H$.

Let $\L_{n}$ be the lattice of subfields of $\K_{n}$. In the case that $n$ is a power of a prime, the lattice of subfields of $\K_n$ is naturally \emph{layered} by the powers of the prime.  Let the \emph{0th layer} of $\L_{p^n}$ be $\L_0 = \{\Q\}$. For $k\ge 1$, the \emph{$k$th layer} of $\L_{p^n}$ is $\L_{p^k}\setminus \L_{p^{k-1}}$. In particular, the layers form a partition of $\L_{p^n}$. The \emph{top layer} of $\L_{p^n}$ is the $n$th layer.  We define the \emph{bottom layer} of $\L_{p^n}$ to be $\L_p$, which is the union of the 1st and 0th layers. 

By \lemref{prop:AutGroups}, the Galois groups of powers of 2 behave differently from the Galois groups of powers of an odd prime. Thus, we must consider the two cases separately. We will address the odd prime case first, followed by the even case.

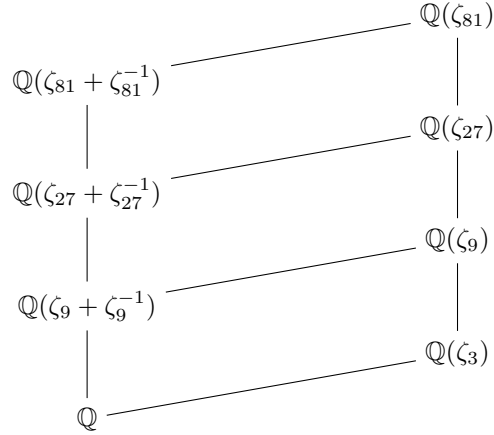
\begin{figure}[htbp]
\caption{The Lattice of Subfields of $\Q(\zeta_{3^4})$.}
\begin{center}
\begin{tikzpicture}[scale= 2.]
\path (0:0cm) node (l10) {$\Q$};
\path (10:2.5cm) node (l11) {$\Q(\zeta_3)$};

\path (l10) +(0,0.75) node (l20) {$\Q(\zeta_9+\zeta_{9}^{-1})$};
\path (l11) +(0,0.75) node (l21) {$\Q(\zeta_{9})$};
\path (l20) +(0,0.75) node (l30) {$\Q(\zeta_{27}+\zeta_{27}^{-1})$};
\path (l21) +(0,0.75) node (l31) {$\Q(\zeta_{27})$};
\path (l30) +(0,0.75) node (l40) {$\Q(\zeta_{81}+\zeta_{81}^{-1})$};
\path (l31) +(0,0.75) node (l41) {$\Q(\zeta_{81})$};

\draw (l10) -- (l11);
\draw (l20) -- (l21);
\draw (l30) -- (l31);
\draw (l40) -- (l41);

\draw (l10) -- (l20) -- (l30) -- (l40);
\draw (l11) -- (l21) -- (l31) -- (l41);
\end{tikzpicture}
\end{center}
\label{fig:81}
\end{figure}

\begin{figure}[htbp]
\caption{The Lattice of Subfields of $\Q(\zeta_{5^4})$.}
\begin{center}
\begin{tikzpicture}[scale= 2.]
\path (0:0cm) node (l10) {$\Q$};
\path (10:2.5cm) node (l11) {$\Q(\zeta_5+\zeta_5^{-1})$};
\path (10:5cm) node (l12) {$\Q(\zeta_5)$};
\path (l10) +(0,0.75) node (l20) {$\Q(\zeta_{25}^{(7)})$};
\path (l11) +(0,0.75) node (l21) {$\Q(\zeta_{25}+\zeta_{25}^{-1})$};
\path (l12) +(0,0.75) node (l22) {$\Q(\zeta_{25})$};
\path (l20) +(0,0.75) node (l30) {$\Q(\zeta_{125}^{(57)})$};
\path (l21) +(0,0.75) node (l31) {$\Q(\zeta_{125}+\zeta_{125}^{-1})$};
\path (l22) +(0,0.75) node (l32) {$\Q(\zeta_{125})$};
\path (l30) +(0,0.75) node (l40) {$\Q(\zeta_{625}^{(182)})$};
\path (l31) +(0,0.75) node (l41) {$\Q(\zeta_{625}+\zeta_{625}^{-1})$};
\path (l32) +(0,0.75) node (l42) {$\Q(\zeta_{625})$};

\draw (l10) -- (l11) -- (l12);
\draw (l20) -- (l21) -- (l22);
\draw (l30) -- (l31) -- (l32);
\draw (l40) -- (l41) -- (l42);

\draw (l10) -- (l20) -- (l30) -- (l40);
\draw (l11) -- (l21) -- (l31) -- (l41);
\draw (l12) -- (l22) -- (l32) -- (l42);
\end{tikzpicture}
\end{center}
\label{fig:625}
\end{figure}

\begin{figure}[htbp]
\caption{The Lattice of Subfields of $\Q(\zeta_{7^4})$.}
\begin{center}
\begin{tikzpicture}[scale= 2.][back line/.style={densely dotted},
				cross line/.style={preaction={draw=white, -,line width=6pt}}]
\path (0:0cm) node (l10) {$\Q$};
\path (5:3cm) node (l11) {$\Q(\zeta_7+\zeta_7^2 + \zeta_7^4)$};
\path (l10) +(45:1cm) node (l12) {$\Q(\zeta_7+\zeta_7^{-1})$};
\path (l11) +(45:1cm) node (l13) {$\Q(\zeta_7)$};

\path (l10) +(0,1.25) node (l20) {$\Q(\zeta_{49}^{(19)})$};
\path (l11) +(0,1.25) node (l21) {$\Q(\zeta_{49}^{(18)})$};
\path (l12) +(0,1.25) node (l22) {$\Q(\zeta_{49}+\zeta_{49}^{-1})$};
\path (l13) +(0,1.25) node (l23) {$\Q(\zeta_{49})$};

\path (l20) +(0,1.25) node (l30) {$\Q(\zeta_{343}^{(19)})$};
\path (l21) +(0,1.25) node (l31) {$\Q(\zeta_{343}^{(18)})$};
\path (l22) +(0,1.25) node (l32) {$\Q(\zeta_{343}+\zeta_{343}^{-1})$};
\path (l23) +(0,1.25) node (l33) {$\Q(\zeta_{343})$};

\path (l30) +(0,1.25) node (l40) {$\Q(\zeta_{2401}^{(1048)})$};
\path (l31) +(0,1.25) node (l41) {$\Q(\zeta_{2401}^{(1047)})$};
\path (l32) +(0,1.25) node (l42) {$\Q(\zeta_{2401}+\zeta_{2401}^{-1})$};
\path (l33) +(0,1.25) node (l43) {$\Q(\zeta_{2401})$};

\draw (l10) -- (l11) -- (l13);
\draw[densely dotted] (l13) -- (l12) -- (l10);
\draw (l20) -- (l21) -- (l23);
\draw[densely dotted] (l23) -- (l22) -- (l20);
\draw (l30) -- (l31) -- (l33);
\draw[densely dotted] (l33) -- (l32) -- (l30);
\draw (l40) -- (l41) -- (l43) -- (l42) -- (l40);

\draw (l10) --  (l20) -- (l30) -- (l40);
\draw (l11) -- (l21) -- (l31) -- (l41);
\draw[densely dotted] (l12) -- (l22) -- (l32) -- (l42);
\draw (l13) -- (l23) -- (l33) -- (l43);
\end{tikzpicture}
\end{center}
\label{fig:343}
\end{figure}

Let us assume that $p$ is an odd prime and let $\G = \G_{p^n}$. Then, by \lemref{prop:AutGroups}, $\G$ is a cyclic group. Thus, the subfields of $\K_{p^n}$ correspond to the divisors of $|\G|$. Now, $|\G| = p^n-p^{n-1} = p^{n-1}(p-1)$. Let $x$ denote the number of divisors of $p-1$. Then $\G$ has $nx$ subgroups. In particular, when $n =1$, $\K_p$ has $x$ subfields, or in other words, $|\L_p| = x$. By induction, the $k$th layer of $\L_{p^n}$ also contains exactly $x$  subfields, for all $1< k\le n$. Furthermore, it follows by induction, elementary properties of cyclic groups, and basic Galois theory that each layer of $\L_{p^n}$ is lattice-isomorphic to the bottom layer, $\L_p$, and each successive layer sits atop its predecessor by a degree $p$ extension, that is, if $\H\le \G_{p^n}$ such that $\K_{p^n}^{\H}$ is in the top layer, then 
\begin{equation} [\K_{p^k}^\H:\K_{p^{k-1}}^\H] = p, \end{equation} for all $1 < k \le n$. Furthermore, 
\begin{equation}\label{prop:sameDegree} \left[\K_{p^n} : \K_{p^n}^{\H}\middle] = \middle[\K_{p^k} : \K_{p^k}^{\H}\right]\end{equation} for all $1 \le k \le n$. These observations give a method to  build the lattice of subfields for $\K_{p^n}$. Figures \ref{fig:81},  \ref{fig:625}, and \ref{fig:343} offer examples of the lattice $\L_{p^n}$ for primes $p=3$, $5$, and $7$, respectively. In Figures  \ref{fig:625} and \ref{fig:343}, the element $\zeta_{n}^{(m)}$ denotes the period of $\zeta_n$ corresponding to the automorphism group generated by $\sigma_m$.

\begin{figure}[htbp]
\caption{The Lattice of Subfields of $\Q(\zeta_{64})$.}
\begin{center}
\begin{tikzpicture}[scale= 2.]

\path (0,0) node (l00) {$\Q$};

\path (l00) +(6,1.5 cm) node (l23) {$\Q(i)$};

\path (l00) +(0,1.5 cm) node (l31) {$\Q(\zeta_8+\zeta_8^{-1})$};
\path (l00) +(3,1.5 cm) node (l32) {$\Q(\zeta_{8}-\zeta_{8}^{-1})$};
\path (l23) +(0,1.5 cm) node (l33) {$\Q(\zeta_8)$};

\path (l31) +(0,1.5 cm) node (l41) {$\Q\left(\zeta_{16}+\zeta_{16}^{-1}\right)$};
\path (l32) +(0,1.5 cm) node (l42) {$\Q\left(\zeta_{16}-\zeta_{16}^{-1}\right)$};
\path (l33) +(0,1.5 cm) node (l43) {$\Q(\zeta_{16})$};

\path (l41) +(0,1.5 cm) node (l51) {$\Q\left(\zeta_{32}+\zeta_{32}^{-1}\right)$};
\path (l42) +(0,1.5 cm) node (l52) {$\Q\left(\zeta_{32}-\zeta_{32}^{-1}\right)$};
\path (l43) +(0,1.5 cm) node (l53) {$\Q(\zeta_{32})$};

\path (l51) +(0,1.5 cm) node (l61) {$\Q\left(\zeta_{64}+\zeta_{64}^{-1}\right)$};
\path (l52) +(0,1.5 cm) node (l62) {$\Q\left(\zeta_{64}-\zeta_{64}^{-1}\right)$};
\path (l53) +(0,1.5 cm) node (l63) {$\Q(\zeta_{64})$};

\draw (l00) -- (l23) -- (l33) -- (l43) -- (l53) -- (l63);
\draw (l00) -- (l31) -- (l41) -- (l51) -- (l61);

\draw (l00) -- (l32) -- (l33); \draw (l31) -- (l42) -- (l43); \draw (l41) -- (l52) -- (l53); \draw (l51) -- (l62) -- (l63);
\draw (l31) -- (l33); \draw (l41) -- (l43);  \draw (l51) -- (l53); \draw (l61) -- (l63);

\end{tikzpicture}
\end{center}
\label{fig:64}
\end{figure}

Next, we will switch our attention to the case when $p=2$. As seen in \lemref{prop:AutGroups}, $\G_{2^n} \cong Z_2\times Z_{2^{n-2}}$ for $n\ge 2$ and $\G_2 = 1$. In particular, $\G_4 \cong Z_2$, and hence $\K_4 = \Q(i)$ contains two subfields: $\Q(i)$ and $\Q$. For $\K_8$, we have that $\G_8 \cong Z_2\times Z_2$, the Klein 4-group. Thus, $\K_8$ has 5 subfields: $\Q$, $\Q(i)$, $\Q(\zeta_8)$, $\Q(\zeta_8+\zeta_8^{-1})$, and $\Q(\zeta_8+\zeta_8^3)$. 

For the fourth layer of the lattice, we notice that $\G_{16} \cong Z_2\times Z_4$ contains a copy of $Z_2\times Z_2$ and hence contains three additional subgroups: two subgroups of order 4 and a subgroup of order 8. Thus, $\L_{16}$ contains three additional fields outside of $\L_8$ by Galois correspondence. These fields are in fact $\Q(\zeta_{16}),$ $\Q(\zeta_{16}+\zeta_{16}^{-1})$, and $\Q(\zeta_{16}+\zeta_{16}^7)$. 

In general, $Z_2\times Z_{2^{n-2}}$ has three more subgroups than $Z_2\times Z_{2^{n-3}}$, and hence $\L_{2^n}$ has three more fields than $\L_{2^{n-1}}$ for $n\ge 3$. One of the fields is certainly $\Q(\zeta_{2^n})$. Since $\G_{2^n}$ is a 2-group, the remaining two fields must correspond to subgroups of $\G_{2^n}$ of order 2. The group $Z_2\times Z_{2^{n-2}}$ has three elements of order 2. In particular, $\sigma_{2^n-1}, \sigma_{2^{n-1}-1}$, and $\sigma_{2^{n-1}+1}$ have order 2 in $\G_{2^n}$. But $\zeta^{2^{n-1}} = -1$. So, for the subgroup $\H = \langle \sigma_{2^{n-1}+1}\rangle$, $\overline{\O_\zeta} = \zeta+\zeta^{2^{n-1}+1} =  \zeta - \zeta = 0$. Also, $\overline{\O_{\zeta^2}} = \overline{\O_{\zeta_{2^{n-1}}}} = \zeta_{2^{n-1}}$. Hence, $\Q(\zeta_{2^{n-1}}) \subseteq \Q(\zeta_{2^n})^\H$, which implies that $\Q(\zeta_{2^n})^\H = \Q(\zeta_{2^{n-1}})$ by degree considerations. Therefore, the additional two fields are $\Q(\zeta+\zeta^{-1})$ and $\Q(\zeta-\zeta^{-1})$. 

Next, we notice that $\Q(\zeta_{2^n})^{\langle \sigma_{2^{n}-1}\rangle}\cap \Q(\zeta_{2^{n-1}}) = \Q(\zeta_{2^n})^{\langle \sigma_{2^{n-1}-1}\rangle}\cap \Q(\zeta_{2^{n-1}}) = \Q(\zeta_{2^{n-1}})^{\langle \sigma_{2^{n-1}-1}\rangle}\linebreak = \Q(\zeta_{2^{n-1}}+\zeta_{2^{n-1}}^{-1})$, by induction. Also, the automorphisms $\sigma_{2^n-1}$ and $\sigma_{2^{n-1}-1}$ are contained in exactly one subgroup of order 4 in $\G_{2^n}$, which corresponds to $\Q(\zeta_{2^{n-1}}+\zeta_{2^{n-1}}^{-1})$.   These observations lead to a method to build $\L_{2^n}$ for any $n$. In particular, each new layer of $\L_{2^n}$ contains three new fields: $\Q(\zeta_{{2^n}})$, $\Q(\zeta_{{2^n}}+\zeta_{{2^n}}^{-1})$, and $\Q(\zeta_{{2^n}}-\zeta_{{2^n}}^{-1})$, where the cyclotomic field $\Q(\zeta_{{2^n}})$ sits directly above the other two fields and $\Q(\zeta_{2^{n-1}})$ and where the fields $\Q(\zeta_{{2^n}}+\zeta_{{2^n}}^{-1})$ and $\Q(\zeta_{{2^n}}-\zeta_{{2^n}}^{-1})$ sit directly above $\Q(\zeta_{2^{n-1}}+\zeta_{2^{n-1}}^{-1})$. To illustrate, we construct $\L_{64}$ in \figref{fig:64}.

\section{Schur Rings}\label{sec:Schur}
Let $S$ be a Schur ring over the finite group $G$ afforded by the partition $\{C_1, C_2, \ldots, C_r\}$. The subsets $C_1, \ldots, C_r$ are called the \emph{primitive sets} of $S$ or \emph{$S$-classes}.  Let $\D(S)$ denote the set of $S$-classes. If $C\subseteq G$ and $\overline{C}\in S$, then $C$ is called an \emph{$S$-set}. If $C$ is also a subgroup of $G$, then we say that $C$ is an \emph{$S$-subgroup} of $G$. If $H$ is an $S$-subgroup, then let 
\[S_H = \Span_F\{\overline{C_i} : C_i \subseteq H\}.\] Then $S_H$ is a Schur ring over $H$.

%
%

Define additional operations on $\FG$ as follows: 
\[* : \FG \to \FG :\quad \left(\sum_{g\in G} \alpha_gg\right)^* = \sum_{g\in G} \alpha_gg^{-1}\] and the \emph{Hadamard product} \[\circ : \FG\times \FG \to \FG: \quad \left(\sum_{g\in G} \alpha_gg\right) \circ \left(\sum_{g\in G}\beta_gg\right) = \sum_{g\in G} \alpha_g\beta_gg.\] Schur rings can then be characterized by these operations. For examples, the subspaces of $\FG$ which are closed under $\circ$ are exactly those spanned by simple quantities.

\begin{Prop}[\cite{Muzychuk94} Lemma 1.3]\label{thm:circleProduct} Suppose that $S$ is a subalgebra of $F[G]$. Then $S$ is a Schur ring if and only if $S$ is closed under $*$ and $\circ$ and $1, \overline{G}\in S$.\end{Prop}

\begin{Cor}\label{prop:intersect} Let $S$ and $T$ be Schur rings over $G$. Then $S\cap T$ is a Schur ring over $G$.\end{Cor}

Every finite group algebra $F[G]$ is a Schur ring, resulting from the partition of singletons on $G$. The partition $\{\{1\}, G\setminus \{1\}\}$ affords a \emph{trivial Schur ring} over $G$, denoted $\FG^0$.

Let $\H\le \Aut(G)$ and \[F[G]^\H = \{\alpha \in F[G] : \sigma(\alpha) = \alpha,\; \text{for all}\; \sigma\in \H\}.\]  Then $\FG^\H$ is a Schur ring afforded by the partition of $G$ corresponding to the orbits of the $\H$-action on $G$. These Schur rings are referred to as \emph{orbit Schur rings}. The center of $\FG$ is an orbit Schur ring with $\H=\Inn(G)$. Let $\rr(\FG) = \FG^{\Aut(G)}$ denote the \emph{rational Schur ring}, whose primitive sets are the automorphism classes of $G$.  For an abelian group $G$, let $\s(\FG) = \FG^{\langle\ast\rangle}$ denote the \emph{symmetric Schur ring}, whose primitive sets are the inverse classes of $G$.

Let $S$ and $T$ be Schur rings over $F[G]$ and $F[H]$, respectively. We naturally can view $G$ and $H$ as subgroups of $G\times H$. Let \[S\cdot T = \Span_F\{\overline{C}\cdot\overline{D} : C\in\D(S), D\in \D(T)\},\] called the \emph{dot product} of $S$ and $T$. This forms a Schur ring with partition $\D(S\cdot T) = \{CD\subseteq G\times H : C\in \D(S), D\in \D(T)\}$. Furthermore, $S\cdot T \cong S\otimes_F T$, as $F$-algebras. Because of this fact, the Schur ring $S\cdot T$ is often called the tensor product of Schur rings.

The notion of \emph{wedge product} of Schur rings presented below is originally based upon the presentation of Leung and Man \cite{LeungI}, although it has been adapted from the original for the purposes of this paper.

Let $H\trianglelefteq G$ and let $S$ be a Schur ring over $G/H$. Let $\pi : G \to G/H$ be the natural quotient map. Consider the partition of $G$ given by $\D = \{\pi^{-1}(C) : C\in \D(S)\},$ that is, if $C = \{g_1H, g_2H, \ldots, g_kH\}\in \D(S)$, then $\pi^{-1}(C) = \bigcup_{i=1}^k g_iH \in \D$.  Let $\pi^{-1}(S) = \Span_F\{\overline{D} : D\in \D\}$. Then $\pi^{-1}(S)$ is a subalgebra over $\FG$ closed under $*$ and $\circ$ and contains $\overline{H}$ and $\overline{G}$, referred to as the \emph{inflated Schur ring} of $S$ over $G$. Let $\D(\pi^{-1}(S)) = \{\pi^{-1}(C) : C\in \D(S)\}$, which is a partition of $G$.

Let $H \trianglelefteq G$, and let $S$ and $T$ be Schur rings over $H$ and $G/H$, respectively. Then \[S\wr T =  S+\pi^{-1}(T),\] called the \emph{wreath product} of $S$ and $T$.  The wreath product is likewise a Schur ring with partition $\D(S\wr T) = \D(S) \cup  (\D(\pi^{-1}(T))\setminus \{H\})$. It follows that $(S\wr T)_H = S$ and $\pi(S\wr T) = T$. 

Let $1 < K \le H <G$  be a sequence of finite groups such that $K\trianglelefteq G$. Let $S$ be a Schur ring over $H$ and $T$ a Schur ring over $G/K$. Let $\pi : G \to G/K$ be the quotient map. Let \[S \wedge_K T = S+\pi^{-1}(T),\] which denotes the \emph{wedge product} of $S$ and $T$. When the context is clear, the subscript may be omitted. If we assume that $H/K$ is a $T$-subgroup, $K$ is an $S$-subgroup, and $\pi(S) = T_{H/K}$, then $S\wedge T$ is a Schur ring over $G$ with partition $D(S\wedge T) = D(S) \cup (\D(\pi^{-1}(T)) \setminus \D(\pi^{-1}(T_{H/K})))$. Like above, it follows that $(S\wedge T)_H = S$ and $\pi(S\wedge T) = T$.  If $H=K$, then $S\wedge T = S\wr T$. Thus, the wedge product of Schur rings is a generalized wreath product of Schur rings.

Let $S$ be a Schur ring over $G$. If there exists subgroups $1 < K \le H < G$, with $K\trianglelefteq G$, and Schur rings $R$ and $T$ over $H$ and $G/K$, respectively, such that $S = R\wedge_K T$, then we say that $S$ is \emph{wedge-decomposable}; otherwise, we say that $S$ is \emph{wedge-indecomposable}. If $S$ is wedge-decomposable, we call $1< K\le H < G$ a \emph{wedge-decomposition} \label{wreath} of $S$. We define the terms \emph{wreath-decomposable}, \emph{wreath-indecomposable}, and \emph{wreath-decomposition} analogously.

Every wreath-decomposable Schur ring is clearly wedge-decomposable and every wedge-indecomposable Schur ring is wreath-indecomposable. On the other hand, there do exist Schur rings which are wreath-indecomposable but wedge-decomposable.



Let $Z_n=\langle z  : z^n\rangle$ denote the cyclic group of order $n$. For each $d\mid n$, let $L_d$ denote all elements of $Z_n$ of order $d$. We will call this the $d$th \emph{layer} of $Z_n$. Each layer is just an automorphism class of $Z_n$. 

Leung and Man used the constructions of Schur rings mentioned above to classify all Schur rings over $Z_n$.

\begin{Thm}[\cite{LeungII, LeungI}]\label{thm:LeungMan} Let $G = Z_n$ and let $S$ be a Schur ring over $G$. Then $S$ is trivial, an orbit ring, a dot product of Schur rings, or a wedge product of Schur rings. \end{Thm}

\begin{Cor}\label{thm:pnLeungMan} Let $G = Z_{p^n}$, for some prime $p$, and let $S$ be a Schur ring over $G$. Then $S$ is trivial, an orbit ring,  or a wedge product of Schur rings. \end{Cor}
\begin{proof}
Since $G$ is a cyclic $p$-group, it cannot be expressed as a nontrivial direct product of groups. As a consequence, $S$ cannot be expressed as a nontrivial dot product of Schur rings. The result then follows from \thmref{thm:LeungMan}.
\end{proof}

\begin{Cor}\label{thm:pLeungMan}\label{thm:primeCycStructure} Let $G = Z_{p}$, for some prime $p$, and let $S$ be a Schur ring over $G$. Then $S$ is an orbit Schur ring. \end{Cor}
\begin{proof}
Since $G$ has no nontrivial subgroups, $S$ is wedge-indecomposable. Also, $\FG^0 = \rr(\FG)$. Thus, the result then follows from \corref{thm:pnLeungMan}.
\end{proof}

\begin{Cor}\label{cor:possibleNucleus} Let $G = Z_{p^n}$. Then for any wedge-decomposable Schur ring $S$ over $G$, there exists a wedge-decomposition $1< K\le H< G$ such that $S_H$ is a wedge-indecomposable orbit algebra or a trivial Schur ring over $H$. 
\end{Cor}
\begin{proof} 
By assumption, $S$ has a wedge-decomposition $1<K\le H< G$. If $S_H$ is wedge-decomposable, then it also has a wedge decomposition $1 < K' \le H' < H$. Since $K$ and $K'$ are nontrivial subgroups of $Z_{p^n}$, $K\cap K'$ is likewise nontrivial. Next, every $S$-class outside of $H$ is a union of cosets of $K$. So, every such $S$-class is also a union of cosets of $K\cap K'$. Similarly, every $S$-class inside of $H$ but outside of $H'$ is a union of cosets of $K\cap K'$. Therefore, $1<K\cap K' \le H' < G$ is a wedge-decomposition of $S$. It follows that a wedge-decomposition $1<K\le H<G$ of $S$ can be chosen such that $H$ is minimal. Such a choice implies that $S_H$ must be wedge-indecomposable. By \corref{thm:pnLeungMan}, $S_H$ is either a trivial or orbit Schur ring.
\end{proof}

\section{A Correspondence Between Schur Rings and Cyclotomic Fields}

Every automorphism of $Z_n$ is determined by $z\mapsto z^m$, and every automorphism on $\K_n$ is similarly determined by $\zeta\mapsto \zeta^m$, where $m$ is unique modulo $n$ and $\gcd(n,m)=1$. Identifying these congruence classes provides an isomorphism between $\Aut(Z_n)$ and the Galois group $\G_n$. 

Let $\omega_n : \Q[Z_n] \to \Q(\zeta_n)$ be the $\Q$-algebra map uniquely defined by the relation $\omega_n(z) = \zeta_n$. Let $S$ be a Schur ring over $\Q[Z_n]$. Then $\omega_n(S)$ is a subalgebra of $\K_n$ and necessarily must be a subfield. From Galois theory, each subfield of $\K_n$ corresponds to a subgroup of $\G_n$. Likewise, each orbit Schur ring corresponds to a subgroup of $\Aut(Z_n) = \G_n$. This defines a one-to-one correspondence between the subfields of $\K_n$ and the orbit Schur rings of $\Q[Z_n]$. In fact, $\omega_n$ determines this correspondence. Since $\Aut(Z_n) = \G_n$, it holds that $\sigma_m\circ \omega_n = \omega_n\circ \sigma_m$.  Hence, if $\H\le \G_n$, then $\omega_n$ preserves $\H$-orbits and $\H$-periods, that is, $\omega_n(\overline{\O_{z^k}}) = \overline{\O_{\zeta^k}}$ for every $k\in \Z$. Since both $\Q[Z_n]^\H$ and $\K_n^\H$ are spanned by their $\H$-periods, we have \begin{equation}\label{prop:RationalCyclicLattice} \omega_n(\Q[Z_n]^\H) = \K_n^\H. \end{equation}

\begin{Prop}\label{cor:distinctRationals} Let $G=Z_n = \langle z\rangle$ and let $\K= \Q(\zeta_n)$. Then the lattice of orbit Schur rings over $G$ is lattice-isomorphic via $\omega_n$ to the lattice of subfields of $\K$.\end{Prop}

\begin{proof}
Equation \eqref{prop:RationalCyclicLattice} shows that $\omega$ is clearly surjective onto the lattice of subfields. If $\omega(\QG^{\H_1}) = \omega(\QG^{\H_2})$ for $\H_1,\H_2\le \G$, then $\K^{\H_1} = \K^{\H_2}$, but the Fundamental Theorem of Galois theory implies that $\H_1 = \H_2$. Therefore, $\omega$ is injective, which proves that $\omega$ is an isomorphism between these two lattices.
\end{proof}

\begin{Cor}\label{cor:distinctRationals2} Let $\H_1, \H_2\le \G_n$. Then $\Q[Z_n]^{\H_1} = \Q[Z_n]^{\H_2}$ if and only if $\H_1=\H_2$. \end{Cor}

According to \corref{cor:distinctRationals2}, distinct automorphism subgroups of $\G_n$ produce distinct orbit Schur rings over $Z_n$.  \corref{cor:distinctRationals2} is not true for arbitrary groups. For example, let $G = Z_4\times Z_2 = \langle a, b\rangle$ and let \[S = \Span_\Q\{1,\; a^2,\; b+a^2b,\; a+a^3+ab+ab^3\} \cong \Q Z_2\wr \Q Z_2 \wr \Q Z_2.\] In fact, $S = \rr(\QG)$. Furthermore, the automorphism group $\Aut(G)$ is given by 
\[\Aut(G) = \left< \sigma : \substack{a\mapsto a\\\\ b\mapsto a^2b},\; \tau :\substack{ a\mapsto a^3b\\\\ b\mapsto a^2b}\right> \cong D_4.\]
 Let $\H = \langle \tau\rangle \lneq \Aut(G)$.  It can be shown that $\QG^\H = \rr(\QG) = \QG^{\Aut(G)}$, although $\H\neq \Aut(G)$.

Since $\Q \subseteq \Q[Z_n]^0 \subseteq \rr(\Q[Z_n])$, \eqref{prop:RationalCyclicLattice} implies that 
\[\Q = \omega(\Q) \subseteq \omega(\Q[Z_n]^0) \subseteq \omega(\rr(\Q[Z_n])) = \Q. \] Therefore,  
\begin{equation}\label{prop:CycTriv} \omega_n(\Q[Z_{n}]^0) = \Q. \end{equation}




If $\Phi_n(x)\in \Z[x]$ denotes the $n$th cyclotomic polynomial, then $\Q(\zeta_n) \cong \Q[x]/(\Phi_n(x))$. Since $\Phi_n(x) \bmid (x^n-1)$ and $\Q[Z_n] \cong \Q[x]/(x^n-1)$, the quotient map $\Q[x] \to \Q(\zeta_n)$ factors  as the composition $\Q[x] \to \Q[Z_n] \xrightarrow{\;\omega_n\;} \Q(\zeta_n)$. In particular, $\ker \omega_n = (\Phi_n(z)) \subseteq \Q[Z_n]$.

\begin{Lem}\label{lem:coset0} For each prime $p$ dividing $n$, let $Z_p$ denote the subgroup of $Z_n$ of order $p$. Then $\ker \omega_{n} = (\overline{Z_p} : p\mid n) = \Span\{\overline{gZ_p} : p\mid n, g\in Z_n\}$. In particular, a simple quantity of $\Q[Z_{n}]$ is a kernel element if and only if it is a sum of unions of cosets of some non-trivial subgroups of $G$. \end{Lem}

\begin{proof}
Let $d\mid n$, and let $f_d(x) = \sum_{k=1}^d x^{n-kn/d}$. Then $\Phi_n(x)$ is the greatest common divisor of $\{f_p(x) : p\mid n,\ p\text{ is prime}\}$.
Therefore, the ideal generated by the $f_p(x)$ is the principal ideal generated by $\Phi_n(x)$. In particular,
\[\ker \omega_n = (\Phi_n(z)) = (f_p(z) : p\mid n) = (\overline{Z_p} : p\mid n). \qedhere\]
\end{proof}

Let $S=S_H\wedge S_{G/K}$ be a wedge-decomposable Schur ring of $\Q[Z_n]$ with wedge decomposition $1< K \le H < G$. If $|H|=m$, then  $\omega_n|_{\QH} = \omega_m$. It then follows that $\omega_m(S_H) =\omega_n(S_H) \subseteq \omega_n(S)$. Conversely, for any $C\in \D(S)\setminus \D(S_H)$, we see that $\overline{C}$ is a union of cosets of $K$ and so $\omega(\overline{C}) =0$, by \lemref{lem:coset0}. Therefore, $\omega(S) \subseteq \omega(S_H)$, which proves that 
\begin{equation}\label{prop:NucleusDetermines} \omega_n(S_H\wedge S_{G/K}) = \omega_m(S_H). \end{equation}

For the remainder of the paper, we will focus on the case when $G = Z_{p^n}$ for a prime $p$.

The kernel $\ker(\omega_{p^n})$ becomes $(\overline{Z_p})$, which is spanned by the cosets of $Z_p$. Since $\ker(\omega_{p^n})$ is spanned by simple quantities, it is closed under the Hadamard product.

\begin{Prop}\label{prop:rationalAll}
Let $G = Z_{p^n}$ for some prime $p$ and let $\rr(\QG) = \QG^{\G}$. Then $\rr(\QG) = \bigwr_{k=1}^n \Q[Z_{p}]^0$.
\end{Prop}

\begin{proof}
Recall that $L_d$ denotes the $d$th layer of $G$, that is, the subset of all elements of order $d$. Then $\rr(\QG) = \Span_\Q\{\overline{L_d} : d\mid p^n\} =  \Span_\Q\{\overline{L_{p^k}} : 1\le k \le n\}$.  For $n=1$, then $\rr(G) = \QG^0$. Assume that the result holds for each $k< n$. For each layer, \[L_{p^k} = \bigcup_{g\in L_{p^k}} gZ_{p^{k-1}},\] that is, $L_{p^k}$ is the union of all nontrivial cosets of $Z_{p^{k-1}}$ in $Z_{p^k}$. Let $\pi : Z_{p^k} \to Z_{p^k}/Z_{p^{k-1}}$ be the natural map. Thus, $\Span\{\overline{Z_{p^{n-1}}}, \overline{L_{p^n}}\} =\pi^{-1}(\Q[Z_{p}]^0)$. Therefore, $\rr(\QG) = \Span\{\overline{L_{p^k}} : 0\le k\le n-1\} \wr \Q[Z_p]^0$. But $ \Span\{\overline{L_{p^k}} : 0\le k\le n-1\} = \rr(\Q[Z_{p^{n-1}}])$. So by induction, \[\rr(G) = \left(\bigwr_{i=0}^{n-1} \Q[Z_p]^0\right) \wr \Q[Z_p]^0 = \bigwr_{k=1}^n \Q[Z_{p}]^0.\qedhere \]
\end{proof}

\begin{Thm}\label{cor:Claim4} Let $S_1$ and $S_2$ be Schur rings over $Z_{p^{n_1}}$ and $Z_{p_{n_2}}$, respectively. Let $n = \max\{n_1, n_2\}$. So, $S_1$ and $S_2$ are subalgebras of $\Q[Z_{p^n}]$. Then \[\omega_n(S_1)\cap \omega_n(S_2) = \omega_n(S_1\cap S_2).\]
\end{Thm}
\begin{proof} For functions, it is always true that 
\[\omega_n(S_1\cap S_2) \subseteq \omega_n(S_1)\cap \omega_n(S_2).\] Suppose that $\omega_n(S_1) = \Q$. Then 
\[\Q\subseteq \omega_n(S_1\cap S_2) \subseteq \omega_n(S_1)\cap\omega_n(S_2) = \Q\cap \omega_n(S_2) = \Q.\] Therefore, $\omega_n(S_1\cap S_2) = \omega_n(S_1) \cap \omega_n(S_2).$\\

Let $i=1,2$. Suppose next that $\omega_n(S_i) = \K_n^{\H_i}$ for $\H_i\le \G_n$. Let $H_i\le Z_{p^{n_i}}$ such that $H_i$ is minimal with respect to the property $\omega_n(S_i) = \omega((S_i)_{H_i})$. By \corref{cor:possibleNucleus} and \eqref{prop:NucleusDetermines}, $(S_i)_{H_i} = \Q[Z_{p^{n_i}}]^{\H_i}$. Let $H = H_1\cap H_2$ and $\H = \H_1\H_2$. Let $m = |H|$. Then 
\begin{eqnarray*}
\omega_n(S_1\cap S_2) & \supseteq& \omega_n((S_1)_{H_1}\cap (S_2)_{H_2}) \\
&=& \omega_n(\Q[Z_{p^{n_1}}]^{\H_1}\cap \Q[Z_{p^{n_2}}]^{\H_2}) = \omega_n(\Q[H]^{\H})\\
& =& \K_m^\H = \K_{m_1}^{\H_1}\cap \K_{m_2}^{\H_2} = \omega_n(\Q[H_1]^{\H_1})\cap \omega_n(\Q[H_2]^{\H_2})\\
&=& \omega_n((S_1)_{H_1})\cap\omega_n((S_2)_{H_2}) = \omega_n(S_1)\cap\omega_n(S_2),
\end{eqnarray*} which finishes the proof.
\end{proof}

\begin{Cor}\label{cor:Chuck2} Let $S$ be a Schur ring over $Z_{p^n}$. If $H=Z_d$ is an $S$-subgroup, then \[\omega(S) \cap \Q(\zeta_{d}) = \omega(S_H).\] \end{Cor}
\begin{proof}
By \thmref{cor:Claim4}, $\omega(S) \cap \Q(\zeta_{d}) = \omega(S) \cap \omega(\QH) = \omega(S\cap \QH) = \omega(S_H)$.
\end{proof}

\begin{Thm}\label{Thm:Claim7} Let $G = Z_{p^n}$ and let $S$ be a Schur ring over $G$ such that $\omega(S) = \Q$. Then there exists a subgroup $H\le G$ and a Schur ring $T$ over $G/H$ such that $S = \Q[H]^0\wr T$.\end{Thm}
\begin{proof}
By \corref{thm:pnLeungMan}, $S$ is trivial, an orbit ring, or wedge-decomposable. If $S$ is trivial, then we are done. In the second case, by \propref{cor:distinctRationals}, $\rr(\QG)$ is the unique Schur ring which maps onto $\Q$, which has the desired form by \propref{prop:rationalAll}. Finally, suppose there exists $S$-subgroups $1 < K \le H < G$ such that $S = S_H \wedge S_{G/K}$ and $S_H$ is trivial or a wedge-indecomposable orbit Schur ring.  But $\omega(S_H) = \omega(S) = \Q$. Since the only indecomposable orbit ring which maps onto $\Q$ is $S_H = \Q[Z_p]^0$, again by \propref{prop:rationalAll}, in either case $S_H = \QH^0$. Since $K$ is a non-trivial $S_H$-subgroup, it must be that $K = H$. Therefore, $S = S_H\wr S_{G/H} = \QH^0\wr S_{G/H}$.
\end{proof}

\begin{Thm}\label{thm:Claim1} Let $S$ be a Schur ring over $Z_{p^n}$  and $\omega(S) \in\L_{p^{n-1}}\setminus \{\Q\}$. Then $S$ is wedge-decomposable. \end{Thm}

\begin{proof}
Let $G=Z_{p^n}$. Since $\omega(S)$ is not in the top layer of $\L_{p^n}$, there exists some subgroup $H\lneq G$ such that $\omega(S) \subseteq \omega(\QH)$ and $H$ is chosen minimally. Then \[\omega(S) = \omega(S) \cap \omega(\QH) = \omega(S\cap \QH) = \omega(S_H).\] By the minimality of $H$, $H$ must be an $S$-subgroup. Since $\omega(S_H)\neq \Q$, $H$ must be a nontrivial subgroup of $G$. In particular, $\overline{G}-\overline{H}\in \ker(\omega|_S)$.

Since $\ker(\omega)$ and $S$ are closed under the Hadamard product, $\ker(\omega|_S)=\ker(\omega)\cap S$ is likewise closed under $\circ$. Let $C\in \D(S)\setminus \D(S_H)$. Since $\omega(S) = \omega(S_H)$, there exists some $\alpha \in S_H$ such that $\omega(\overline{C}) = \omega(\alpha)$. Thus, $\overline{C}-\alpha\in \ker(\omega|_S)$. Therefore, $\overline{C} = (\overline{C}-\alpha)\circ (\overline{G}-\overline{H}) \in \ker(\omega|_S)$. Since $\overline{C}$ is a simple quantity, \lemref{lem:coset0} implies that $C$ is a union of cosets of some subgroup. This subgroup is exactly the maximal subset of $G$ which stabilizes $C$. A result due to Wielandt \cite{Wielandt64} states that these stabilizers are $S$-subgroups. Taking intersections if necessary, every class $C\in \D(S)\setminus\D(S_H)$ is a union of cosets of some nontrivial $S$-subgroup $K$. Therefore, $S$ is wedge-decomposable.
\end{proof}

\begin{Thm}\label{thm:topLayerRational} Let $S$ be a Schur ring over $Z_{p^n}$ such that $\omega(S)$ is in the top layer of $\L_{p^n}$. Then $S$ is an orbit Schur ring, and hence $S$ is the unique Schur ring over $Z_{p^n}$ which maps to $\omega(S)$.
\end{Thm}
\begin{proof}
By \eqref{prop:CycTriv} and \eqref{prop:NucleusDetermines} if $S$ is trivial or wedge-decomposable then  $\omega(S)$ is not in the top layer. Thus, $S$ is an orbit Schur ring. By \propref{cor:distinctRationals}, $S$ is the unique orbit Schur ring mapping onto $\omega(S)$.
\end{proof}

\section{Counting Schur Rings over Cyclic $p$-Groups}\label{chap:Counting}

Using the representation $\omega_n : \Q[Z_n] \to \Q(\zeta_n)$ which was considered in the previous chapter, we will construct a recursive formula and generating function for the integer sequence counting the number of Schur rings over $Z_{p^n}$, for $p$ a odd prime. 

\begin{Def}\label{def:Omega} Let $\Omega(n)$ denote the number of Schur rings over $Z_{p^n}$ and let $\Omega(n,k)$ denote the number of Schur rings $S$ over $Z_{p^n}$ such that $\omega(S) = \K_{p^k}$. 
\end{Def} 

We have that $\Omega(0) = 1$ since there is exactly one Schur ring over $Z_{p^0} = 1$, the group ring itself. Also, if $x$ denotes the number of divisors of $p-1$, then $\Omega(1) = x$ by \corref{thm:primeCycStructure}.

\begin{Prop}\label{prop:formulaZeroLayer} The number of Schur rings  over $Z_{p^n}$, for $n \ge 1$, mapping onto $\Q$ with respect to $\omega$ is equal to the sum of the number of Schur rings over $Z_{p^k}$ for $0\le k\le n-1$, that is, 
\begin{equation}\label{eq:formulaZeroLayer}\Omega(n,0) = \sum_{k=0}^{n-1} \Omega(k).\end{equation}   
\end{Prop}

\begin{proof}
Let $G = Z_{p^n}$. By \thmref{Thm:Claim7}, if $\omega(S) = \Q$ then $S = \Q[Z_{p^k}]^0 \wr T$ for some Schur ring $T$ over $G/Z_{p^k}$. If we consider the trivial Schur ring on $G$ as a trivial wreath product, that is, $\QG^0 = \QG^0\wr \Q[1]$, then every Schur ring descending to $\Q$ has the form \[S = \Q[Z_{p^k}]^0\wr T,\] where $1\le k\le n$ and $T$ ranges over all the Schur rings of $G/Z_{p^k} \cong Z_{p^{n-k}}$.  Since every Schur ring over $G$ of this form maps to $\Q$, the proof is finished.
\end{proof}

\begin{Prop}\label{prop:formulaFirstLayer} The number of Schur rings over $Z_{p^n}$ mapping to $\K_p$ with respect to $\omega$ is equal to the number of Schur rings over $Z_{p^{n-1}}$, that is, 
\begin{equation}\label{eq:formulaFirstLayer}\Omega(n,1) = \Omega(n-1).\end{equation} 
\end{Prop}

\begin{proof}
Let $G = Z_{p^n}$. If $n = 1$, then $\Omega(n-1) = \Omega(0) = 1$. By \corref{thm:primeCycStructure}, there is only one Schur ring which maps to $\K_p$. So the result follows.

Suppose that $n\ge 2$.  Let $S$ be the orbit Schur ring over $G = Z_{p^n}$ which maps onto $\K_p$. By \thmref{thm:Claim1}, $S$ is wedge-decomposable.  By \corref{cor:possibleNucleus}, there is a wedge decomposition of $S$, $1< K \le H < Z_{p^n}$, such that $S_H$ is trivial or an indecomposable orbit Schur ring. By \propref{prop:NucleusDetermines}, $\omega(S) = \omega(S_H)$. If $S_H$ is trivial, then $\omega(S_H) = \Q$, by \propref{prop:CycTriv}. Thus, $S_H$ is an indecomposable orbit Schur ring. Now, if $Z_p \neq H$, then $S_H$ is wedge-decomposable by \thmref{thm:Claim1}. Therefore, $H=Z_p$, which forces $K = H$. In fact, $S_H = \Q[Z_p]$. This shows that $S = \Q[Z_p]\wr T$, where $T$ is some Schur ring over $G/Z_p$. Since every Schur ring over $G$ of this form maps to $\K_p$, the proof is finished.
\end{proof}

\begin{Prop}\label{prop:formulaTopLayer} The number of Schur rings over $Z_{p^n}$ mapping to $\K_{p^n}$ with respect to $\omega$ is one, that is,
\begin{equation}\label{eq:formulaTopLayer}\Omega(n,n) = 1.\end{equation}
\end{Prop}

\begin{proof}
Since $\K_{p^n}$ is a field in the top layer, this formula follows immediately from \thmref{thm:topLayerRational}.
\end{proof}

\begin{Prop}\label{prop:formulaLowerLayers} For $n \ge 2$, the number of Schur rings over $Z_{p^n}$ mapping to $\K_{p^k}$ for $1 < k \le n$ with respect to $\omega$ is equal to the sum of the number of Schur rings over $Z_{p^{n-1}}$ mapping onto $\K_{p^j}$ where $j$ ranges between $k-1$ and $n-1$, that is, 
\begin{equation}\label{eq:formulaLowerLayers}\Omega(n,k) = \sum_{j=k-1}^{n-1} \Omega(n-1, j).\end{equation}
\end{Prop}
\begin{proof}
Let $G = Z_{p^n}$. If $k = n$, then $\Omega(n,n) = 1 = \Omega(n-1,n-1)$, by \eqref{eq:formulaTopLayer}.  If $1 < k < n$, then each Schur ring mapping onto $\K_{p^n}$ is wedge-decomposable, by \thmref{thm:Claim1}. In particular, if $S$ is a Schur ring over $Z_{p^n}$ such that $\omega(S) = \K_{p^k}$, then there exists a wedge-decomposition such that $1 < K \le H= Z_{p^k} < G$ and $S_H = \Q[H]$. Put another way, $S = \Q[H] \wedge T$, where $T$ is a Schur ring over $G/K$. Clearly, $\overline{K} \in \Q[H]$ for any choice of $K$. If $\pi : G \to G/K$ is the quotient map, then $\pi(\Q[H]) = \Q[H/K]$. Therefore, the wedge product $\Q[H]\wedge T$ is possible if and only if $H/K$ is a $T$-subgroup and $T_{H/K} = \Q[H/K]$. Without the loss of generality, we may assume that $K= Z_p$, since any coset of $K$ is necessarily a coset of $Z_p$. If we identify $\pi$ with the map $\pi : Z_{p^n} \to Z_{p^{n-1}}$, then $\pi(H) = Z_{p^{k-1}}$ and we must determine which Schur rings $T$ have the property that $T_{H/K} = \Q[Z_{p^{k-1}}]$.  Now, $\omega(T_{H/K}) = \K_{p^{k-1}}$, but by \corref{cor:Chuck2} we have $\omega(T_{H/K}) = \omega(T)\cap \K_{p^{k-1}}$. Equation \eqref{cor:subfieldPeriod} then gives that $\omega(T) = \K_{p^j}$ for some $k-1 \le j \le n-1$. Since every Schur ring of this type can be wedged to $\QH$, the equality is proven.
\end{proof}

\begin{Prop}\label{prop:formulaSameLayer} Let $E, F \in \L_{p^k} \setminus \L_{p^{k-1}}$. Then the number of Schur rings over $Z_{p^n}$ which map onto $E$ with respect to $\omega$ is equal to the number of Schur rings over $Z_{p^n}$ which map onto $F$ with respect to $\omega$. In particular, the number of Schur rings mapping onto $E$ is equal to $\Omega(n,k)$.
\end{Prop}

Remember that $\Q \in \L_{p^0}$ and is contained in the $0$th layer of $\L_{p^n}$, not the first layer $\L_{p} \setminus \L_1$.

\begin{proof}
Let $\Omega(n,E)$ be the number of Schur rings over $Z_{p^n}$  which map onto $E$. If $k =0$, then the only field in this layer is $\Q$. So, $E=\Q$. If $k= 1$, we can mimic the proof of \propref{prop:formulaFirstLayer} to get $\Omega(n,E) = \Omega(n-1) = \Omega(n,1)$. So, we may suppose that $k \ge 2$.

We will now induct on $n$. Let $n = 2$. Then the only $k$ to consider is $k=2$, which represents the top layer. Mimicking the proof of \propref{prop:formulaTopLayer}, we get $\Omega(n,E) = 1 = \Omega(n,k)$. Suppose now that the result holds for all integers less than $n$. Mimicking the the proof of \propref{prop:formulaLowerLayers} (using here also \eqref{prop:sameDegree}), we have
\[\Omega(n,E) = \sum_{j=k-1}^{n-1} \Omega(n-1, E\cap \K_{p^j}).\] By induction, $\Omega(n-1, E\cap \K_{p^j}) = \Omega(n-1, j)$ for each $j$, which proves $\Omega(n,E) = \Omega(n,k)$.
\end{proof}

\begin{Thm} The number of Schur rings over $Z_{p^n}$, where $p$ is an odd prime and $n \ge 2$, is given by the following  equation:
\begin{equation}\label{eq:formula1st} \Omega(n) = \Omega(n,0) + (x-1)\Omega(n,1) + x\sum_{k=2}^n \Omega(n,k),\end{equation} where $x$ denotes the number of divisors of $p-1$.
\end{Thm} 
\begin{proof}
There is exactly one field in the $0$th layer, $(x-1)$ fields in the first layer, and $x$ fields in all remaining layers of $\L_{p^n}$. The equation then follows from \propref{prop:formulaSameLayer}.
\end{proof}

Equation \eqref{eq:formula1st} provides for us a formula which can calculate the number of Schur rings over $Z_{p^n}$ using $\Omega(n,k)$ for $k\le n$. This then begs the question, ``How does one compute $\Omega(n,k)$?'' Equations \eqref{eq:formulaZeroLayer}, \eqref{eq:formulaFirstLayer}, and \eqref{eq:formulaTopLayer} provides answers to this question when $k=0$, $1$, and $n$. For example, we can use \eqref{eq:formula1st} to compute $\Omega(2)$:
\begin{eqnarray*}
\Omega(2) &=& \Omega(2,0) + (x-1)\Omega(2,1) + x\Omega(2,2)\\
&=& (\Omega(0) + \Omega(1)) + (x-1)\Omega(1) + x\\
&=& (1+x) + (x-1)x + x\\
&=& x^2+x+1.
\end{eqnarray*} Using \eqref{eq:formulaLowerLayers}, we can compute all remaining values of $\Omega(n,k)$ recursively. We provide a few examples below.

\begin{Cor} For $n \ge 2$, 
\begin{equation}\Omega(n,n-1) = x + (n-2).\end{equation}
\end{Cor}
\begin{proof}
We proceed by induction on $n$. For $n=2$, we have $\Omega(2,1) = \Omega(1) = x = x + (2-2)$. For $n > 2$, we have 
\begin{eqnarray*}
\Omega(n,n-1) &=& \Omega(n-1, n-2) + \Omega(n-1,n-1)\quad\text{by \eqref{eq:formulaLowerLayers}},\\
&=& \Omega(n-1, (n-1)-1) + 1\quad\text{by \eqref{eq:formulaTopLayer}},\\
&=& x+(n-3) + 1\quad\text{by induction},\\
&=& x + (n-2). \qquad\qquad\qquad\qquad\qquad\qquad\qedhere
\end{eqnarray*}
\end{proof}

\begin{Cor} For $n \ge 3$, 
\begin{equation}\Omega(n,n-2) = x^2 + (n-2)x + \dbinom{n-1}{2}.\end{equation}
\end{Cor}
\begin{proof}
We proceed by induction on $n$. For $n=3$, we have $\Omega(3,1) = \Omega(2) = x^2+x+1 = x^2 + (3-2)x + \dbinom{3-1}{2}$. For $n > 3$, we have 
\begin{eqnarray*}
\Omega(n,n-2) &=& \Omega(n-1, n-3) + \Omega(n-1, n-2) +  \Omega(n-1,n-1)\\
&=& \Omega(n-1, (n-1)-2) + \Omega(n-1, (n-1)-1) + 1\\
&=& \left(x^2+(n-3)x + \dfrac{(n-3)(n-2)}{2}\right) + (x+ (n-3)) + 1\\
&=& x^2 + (n-2)x + \dbinom{n-1}{2}.   \qquad\qquad\qquad\qquad\qquad\qquad\qedhere
\end{eqnarray*}
\end{proof}

By a similar induction argument, we can also prove the identity
\begin{equation}\Omega(n,n-3) = x^3+(n-2)x^2+\left(\dbinom{n-1}{2} + 1\right)x + \left(\dbinom{n}{3}-3\right)\end{equation} for $n\ge 4$. As in the previous proofs, the base case of the induction argument uses the calculation of $\Omega(3)$, which can be computed using $\Omega(3,3)$, $\Omega(3,2)$, $\Omega(3,1)$ and $\Omega(3,0).$ Thus, $\Omega(n)$ can be computed using $\Omega(n,k)$, which can be computed using $\Omega(j)$ for $j < n$. Therefore, there is a recursive procedure to compute $\Omega(n)$ from $\Omega(j)$ for $j<n$. We now will work to unearth this recursive formula.

\begin{figure}[htbp]
\caption{The first several values of $\Omega(n,k)$}
\begin{center}
\begin{enumerate}[$\Omega(1,\cdot) = $]
\item 1\\
\item  $x$, $1$\\
\item $x^2+x+1$, $x+1$, $1$\\
\item $x^3+2x^2+4x+1$, $x^2+2x+3$, $x+2$, $1$\\
\item $x^4+3x^3+8x^2+9x+2$, $x^3+3x^2+7x+7$, $x^2+3x+6$, $x+3$, $1$
\item $x^5+4x^4+13x^3+23x^2+25x+3$, $x^4+4x^3+12x^2+20x+9$, $x^3+4x^2+11x+17$, $x^2+4x+10$, $x+4$, $1$\\
\item $x^6+5x^5+19x^4+44x^3+72x^2+69x+5$, $x^5+5x^4+18x^3+40x^2+61x+54$, $x^4+5x^3+17x^2+36x+51$, $x^3+5x^2+16x+32$, $x^2+5x+15$, $x+5$, $1$\\
\item $x^7+6x^6+26x^5+73x^4+152x^3+222x^2+203x+8$, $x^6+6x^5+25x^4+68x^3+135x^2+188x+163$, $x^5+6x^4+24x^3+63x^2+119x+158$, $x^4+6x^3+23x^2+58x+109$, $x^3+6x^2+22x+53$, $x^2+6x+21$, $x+6$, $1$ 
\end{enumerate}
\end{center}
\end{figure}

Using \eqref{eq:formulaZeroLayer} and \eqref{eq:formulaFirstLayer}, we can rewrite \eqref{eq:formula1st} as
\begin{equation}\label{eq:formula2nd} \Omega(n) = x\Omega(n-1) + \sum_{k=0}^{n-2} \Omega(k) + x\sum_{k=2}^n \Omega(n,k).\end{equation}
Thus, we need to expand $\sum_{k=2}^n \Omega(n,k)$ using \eqref{eq:formulaLowerLayers}. This will produce an equation of the following form:
\begin{equation}\label{eq:CatalanCoefficients} \sum_{i=2}^n \Omega(n,i) = \sum_{i=1}^{n-1} c_{i}\Omega(n-i, 1) = \sum_{i=2}^{n} c_{i-1}\Omega(n-i) \end{equation} for some positive integers $c_i$. In particular, the $j$th iteration of \eqref{eq:formulaLowerLayers} will produce an equation of the form 
\begin{equation}\label{eq:CatalanCoefficientsTriangle} \sum_{k=2}^n \Omega(n,k) = \sum_{i=1}^{j-1} c_{i}\Omega(n-i,1) + \sum_{k=j+1}^{n} c_{jk} \Omega(n-j, k-j)\end{equation} for some positive integers $c_{jk}$. We note that $c_{i(i+1)} = c_i$ and $c_{0k} = 1$ for all $k$. Furthermore, 
\begin{equation}\label{eq:CatalanTriangle1} c_{jk} = \sum_{\ell=j}^{k} c_{(j-1)\ell}\end{equation} by \propref{prop:formulaLowerLayers}. 
When $0 < j < k-1$, \eqref{eq:CatalanTriangle1} can be rewritten recursively to give 
\begin{equation}\label{eq:CatalanTriangle2} c_{jk} = c_{(j-1)k} + \sum_{\ell=j}^{k-1} c_{(j-1)\ell} = c_{(j-1)k} + c_{j(k-1)}.\end{equation}  From \eqref{eq:CatalanTriangle2}, we can create a triangular array of integers, depicted in \figref{tab:CatalanTriangleSmall}, where $k$ indexes the rows ($k \ge 1$) and $j$ indexes the columns ($0\le j < k$). The diagonal entries of the triangle give the values of $c_i$.

\begin{figure}[htbp]
\begin{center}
\caption{The Triangular Array of $c_{jk}$ Coefficients }
\begin{tabular}{cccccccc}
 1  \\ 
 1 & 2  \\ 
 1 & 3 & 5 \\ 
 1 & 4 & 9 & 14 \\ 
 1 & 5 & 14 & 28 & 42 \\ 
 1 & 6 & 20 & 48 & 90 & 132  \\ 
 1 & 7 & 27 & 75 & 165 & 297 & 429   \\ 
1 & 8 & 35 & 110 & 275 & 572 & 1001 & 1430     \\ 
\end{tabular}
\label{tab:CatalanTriangleSmall}
\end{center}
\end{figure}

\begin{Lem}\label{lem:formulaCatalan} Let $c_i$ be the coefficients given in Equation \eqref{eq:CatalanCoefficients}. Then $c_i = \dfrac{1}{i+1}\dbinom{2i}{i}$, that is, $c_i$ is the $i$th Catalan number.
\end{Lem}

\begin{proof}
For convenience, we define $c_{00} =1$ and $c_{jj} = c_{(j-1)j}$ for $j > 0$. This extended triangular array is known as Catalan's Triangle. One property of Catalan's Triangle is that the sequence of diagonal entries is the sequence of Catalan numbers \cite{CatalanTriangle}.
\end{proof}

\begin{Thm}\label{thm:formula3rd} The number of Schur rings over $Z_{p^n}$, where $p$ is an odd prime and $n \ge 1$, is given by the following recursive equation:
\begin{equation}\label{eq:formula3rd}\Omega(n) = x\Omega(n-1) + \sum_{k=2}^n(c_{k-1}x + 1)\Omega(n-k),\end{equation} where $\Omega(0) =1$, $\Omega(1) = x$ denotes the number of divisors of $p-1$, and $c_k = \dfrac{1}{k+1}\dbinom{2k}{k}$ is the $k$th Catalan number.
\end{Thm} 

For $n=1$, we are considering the sum in \eqref{eq:formula3rd} to be empty.

\begin{proof}
The statements $\Omega(0)=1$ and $\Omega(1)=x$ have already been proven. For $n\ge 2$,  
\begin{eqnarray*}
\Omega(n) &=&x\Omega(n-1) + \sum_{k=0}^{n-2} \Omega(k) + x\sum_{k=2}^n \Omega(n,k), \quad\text{by \eqref{eq:formula2nd},}\\
&=& x\Omega(n-1) + \sum_{k=0}^{n-2} \Omega(k) + x\sum_{k=2}^n c_{k-1}\Omega(n-k), \quad\text{by \eqref{eq:CatalanCoefficients},}\\
&=&  x\Omega(n-1) + \sum_{k=2}^n(c_{k-1}x + 1)\Omega(n-k).
\end{eqnarray*} Finally, the formula follows from \lemref{lem:formulaCatalan}.
\end{proof}

By \eqref{eq:formula3rd}, $\Omega(n)$ can be computed recursively without reference to $\Omega(n,k)$ and makes for a much more efficient recurrence. The first several values of $\Omega(n)$ are listed in \figref{tab:OmegaPolynomial}. Now, $\Omega(n)$ is a polynomial of $x$. Thus, the number of Schur rings over $\Z_{p^n}$ is computed by evaluating this polynomial for a specific value of $x$ which depends on the prime $p$. \tableref{table:oddNumberSRings} lists the number of Schur rings over $Z_{p^n}$ up to the tenth power for the first seven odd primes.

\begin{figure}[hbtp]
\caption{The first several $\Omega$-polynomials}
\begin{center}
\begin{enumerate}[$\Omega(1) =$]
\item $x$\\
\item $x^2+x + 1$\\
\item $x^3+2x^2 + 4x + 1$\\
\item $x^4 + 3x^3+8x^2+9x + 2$\\
\item $x^5+4x^4+13x^3+23x^2+25x+3$\\
\item $x^6 + 5x^5+19x^4+44x^3+72x^2+69x+5$\\
\item $x^7 + 6x^6+26x^5+73x^4+152x^3+222x^2+203x + 8$\\
\item $x^8+7x^7+34x^6+111x^5+275x^4+511x^3+703x^2 + 623x+13$\\
\item $x^9 + 8x^8 + 43x^7+159x^6+452x^5+997x^4+1725x^3  +2272x^2+1990x+21$\\
\item $x^{10} + 9x^9 + 53x^8 +218x^7+ 695x^6 +1754x^5+ 3572x^4  +5854x^3 +7510x^2 + 6559x+ 34$
\end{enumerate}
\end{center}
\label{tab:OmegaPolynomial}
\end{figure}

\begin{table}[hbtp]
\caption{Number of Schur Rings over $Z_{p^k}$}
\begin{center}
\begin{tabular}{|c|c|c|c|c|c|c|c|}
\hline
\textbf{$k \backslash p$} & \textbf{3} & \textbf{5} & \textbf{7} & \textbf{11} & \textbf{13} & \textbf{17} & \textbf{19} \\\hline
\textbf{1} & 2 & 3 & 4 & 4 & 6 & 5 & 6\\ \hline
\textbf{2} & 7 & 13 & 21 & 21 & 43 & 31 & 43\\ \hline
\textbf{3} & 25 & 58 & 113 & 113 & 313 & 196 & 313 \\ \hline
\textbf{4} & 92 & 263 & 614 & 614 & 2,288 & 1,247 & 2,288\\ \hline
\textbf{5} & 345 & 1,203 & 3,351 & 3,351 & 16,749 & 7,953 & 16,749\\ \hline
\textbf{6} & 1,311 & 5,531 & 18,329 & 18,329 & 122,675 &50,775 & 122,675\\ \hline
\textbf{7} & 5,030 & 25,511 & 100,372 & 100,372 & 898,706 & 324,323 & 898,706\\ \hline
\textbf{8} & 19,439 & 117,910 & 550,009 & 550,009 & 6,584,443 & 2,072,078 & 6,584,443\\ \hline
\textbf{9} & 75,545 & 545,730 & 3,015,021 & 3,015,021 & 48,243,393 & 13,239,896 & 48,243,393\\ \hline
\textbf{10} & 294,888 & 2,528,263 & 16,531,326 & 16,531,326 & 353,479,684 & 84,603,579 & 353,479,684\\ \hline
\end{tabular}
\end{center}
\label{table:oddNumberSRings}
\end{table}

Examining \figref{tab:OmegaPolynomial}, one can recognize a few patterns with these polynomials. First, $\Omega(n)$ is always a monic degree $n$ polynomial. Next, the coefficient of $x^{n-1}$ is always $n-1$. Both of these statements can be easily proven by induction. Other statements about the coefficients of $\Omega(n)$ can also be stated and proven. Perhaps the most surprising sequence of coefficients is the sequence of constant terms.

\begin{Cor} Let $p$ be an odd prime. Then let $f_n(x) = \Omega(n) \in \Z[x]$. Then $f_n(0) = F_{n-1}$, where $F_n$ is the $n$th term of the Fibonacci sequence.
\end{Cor}

\begin{proof} 
First, we claim that $F_{n} = 1+\sum_{k=0}^{n-2} F_k$ for $n \ge 2$. For $n = 2$, we get $F_2 = 1+F_0 = 1+0 = 1$. For $n > 2$, we get $F_n = F_{n-1} + F_{n-2} = \left(1+\sum_{k=0}^{n-3} F_k\right) + F_{n-2} = 1 + \sum_{k=0}^{n-2} F_k$, which proves the claim. 

It is easy enough to see that $f_1(0) = 0 = F_0$ and $f_2(0) = 1 = F_1$. Suppose that $f_k(0) = F_{k-1}$ for all $k < n$. By \eqref{eq:formula3rd}, 
\[f_n(0) = \sum_{k=2}^n f_{n-k}(0) = \sum_{k=0}^{n-2} f_{k}(0) = 1+ \sum_{k=1}^{n-2} f_{k}(0) = 1+ \sum_{k=1}^{n-2} F_{k-1} = 1+ \sum_{k=0}^{n-3} F_{k} = F_{n-1}.\qedhere\]
\end{proof}

Let $\mathcal F(z) = \sum_{n=0}^\infty \Omega(n)z^n$ be the generating function of $\Omega$. Using \[\mathcal C(z) = \sum_{n=0}^\infty c_nz^n = \dfrac{1-\sqrt{1-4z}}{2z},\] the generating function for the Catalan numbers, and the usual generating function calculations, one computes that 
\begin{equation} \mathcal F(z)  =\dfrac{2(1-z)}{-2z^2+(x-2)z-(x-2) + x(1-z)\sqrt{1-4z}} \end{equation}
Now, one can continue working with the generating function of $\Omega(n)$ using the typical combinatorial methods to produce a non-recursive formula for $\Omega(n)$. Unfortunately, this formula is highly complicated, so we will be content with \eqref{eq:formula3rd}.

\section{Counting Schur Rings Over Cyclic $2$-groups}\label{sec:CountEven}
As is common practice, the case $p=2$ must be treated separately from all other primes as it is the only exceptional case. This section is dedicated to the treatment of Schur rings over $Z_{2^n}$. 

As in the odd case, we mention that the notation introduced in  \defref{def:Omega} applies for $p=2$ also. There are two critical differences between $\L_{2^n}$ and $\L_{p^n}$, for $p$ odd, that should be mentioned. First, there is no first layer on $\L_{2^n}$ since $\L_2 = \L_1 = \{\Q\}$. This will cause our recurrence relation on $\Omega(n)$ to have ``extra'' initial conditions, that is, the recursion does not stabilize until the fourth stage, as opposed to the second stage for odd primes. Second, the Galois group of $\K_{2^n}$ is not cyclic for $n \ge 3$. This gives the lattice $\L_{2^n}$ a different shape than the other lattices we have seen, which translates to different recurrence relations on $\Omega(n,k)$, which we will see below.

Despite these differences, there are still some important similarities between the even and odd cases. For example, it still holds that $\Omega(0) = 1$. Another similarity is the fact that $\Omega(1) = 1$, which is the number of divisors of $2-1=1$. It also holds that \propref{prop:formulaZeroLayer}, \propref{prop:formulaTopLayer}, and \propref{prop:formulaLowerLayers} (for all $n\ge 3$ and $2 < k \le n$) remain true if $p=2$ by the same proofs as before. From these, we can compute 
\[\Omega(2) = \Omega(2,0) + \Omega(2,2) = (\Omega(0) + \Omega(1)) + 1 = 3.\] Now, \propref{prop:formulaFirstLayer} no longer applies since there is no first layer. Instead, we will treat $k=2$ as the base case in the recurrence relation on $\Omega(n,k)$.

\begin{Prop}\label{prop:formulaSecondLayer}  For $n\ge 3$, the number of Schur rings over $Z_{2^n}$ mapping to $\K_4=\Q(i)$ with respect to $\omega$ is equal to the difference between number of Schur rings over $Z_{2^{n-1}}$ and the number of Schur rings over $Z_{2^{n-2}}$ mapping onto $\Q$, that is, 
\begin{equation}\label{eq:formulaSecondLayer} \Omega(n,2) = \Omega(n-1) - \Omega(n-2,0). \end{equation}
\end{Prop}
When $n=2$, we have $\Omega(2,2) = 1$ by \eqref{eq:formulaTopLayer}.
\begin{proof}
Let $S$ be a Schur ring over $Z_{2^n}$ such that $\omega(S) = \Q(i)$. Since $n\ge 3$, it must be that $S$ is wedge-decomposable of the form 
$S = \Q[Z_4]\wedge T$ for some Schur ring $T$ over $Z_{2^{n-1}}$ such that $T\cap \Q[Z_2] = \Q[Z_2]$, by the same reasoning used in \propref{prop:formulaLowerLayers}. Now, every Schur ring over $Z_{2^{n-1}}$ has this property except those of the form $T = \Q[Z_{2^k}]^0\wr T'$ for $1< k \le n-1$. Now, there are exactly $\Omega(n-2,0)$ such Schur rings by \propref{prop:formulaZeroLayer}. Therefore, the result follows.\end{proof}

The major consequence of $\G_{2^n}$ not being cyclic is that \propref{prop:formulaSameLayer} fails for some of the layers of $\L_{2^n}$. For example, the number of Schur rings over $Z_{16}$ which map onto $\K_8=\Q(\zeta_8)$ is three but the number of Schur rings mapping onto $\Q(\zeta_8+\zeta_8^{-1})$ is four. By \secref{sec:cycFields}, for $k\ge 3$, the $k$th layer of $\L_{2^n}$ contains three fields: $\Q(\zeta_{2^k})$, $\Q(\zeta_{2^k}+\zeta_{2^k}^{-1})$, and $\Q(\zeta_{2^k}-\zeta_{2^k}^{-1})$. Let $\Omega_\s(n,k)$ be the number of Schur rings over $Z_{2^n}$ which map onto $\Q(\zeta_{2^n}+\zeta_{2^n}^{-1})$ via $\omega$. It holds that $\s(Z_{2^n})$ is the unique Schur ring over $Z_{2^n}$ which maps onto $\Q(\zeta_{2^n}+\zeta_{2^n}^{-1})$, by \thmref{thm:topLayerRational}. This gives the following formula:
\begin{equation}\label{eq:SformulaTopLayer} \Omega_\s(n,n)=1. \end{equation} Likewise, $\Q[Z_{2^n}]^{\langle \sigma_{2^{n-1}-1}\rangle}$ is the unique Schur ring over $Z_{2^n}$ which maps onto $\Q(\zeta_{2^n}-\zeta_{2^n}^{-1})$. So for the top layer, the number of Schur rings mapping onto a given field is constant. This fact allows use to compute $\Omega(3)$:
\[\Omega(3) = \Omega(3,0) + \Omega(3,2) + 3\Omega(3,3) = (\Omega(0) + \Omega(1) + \Omega(2)) + (\Omega(2) - \Omega(0)) + 3 = 10.\]

Although \propref{prop:formulaSameLayer} is false in general for $p=2$, it is still ``mostly'' true, as explained in the next proposition.

\begin{Prop}\label{prop:2formulaSameLayer} The number of Schur rings over $Z_{2^n}$ mapping onto $\Q(\zeta_{2^k}+\zeta_{2^k}^{-1})$ via $\omega$ is the same as the number of Schur rings mapping onto $\Q(\zeta_{2^k}-\zeta_{2^k}^{-1})$.
\end{Prop}
\begin{proof} If $\Q(\zeta_{2^k}+\zeta_{2^k}^{-1})$ and $\Q(\zeta_{2^k}-\zeta_{2^k}^{-1})$ are in the top layer, then there is exactly one Schur ring mapping onto each field by \eqref{eq:SformulaTopLayer}. Otherwise, each Schur ring mapping onto these fields must be wedge-decomposable. Let $\pi : Z_{2^n} \to Z_{2^{n-1}}$ be the natural quotient map. Then $\pi(\s(Z_{2^k})) = \pi(\Q[Z_{2^k}]^{\langle \sigma_{2^k-1}\rangle}) = \s(Z_{2^{k-1}}) = \pi(\Q[Z_{2^k}]^{\langle \sigma_{2^{k-1}-1}\rangle})$. Since the images are the same, the number of possible wedge products which map on $\Q(\zeta_{2^k}+\zeta_{2^k}^{-1})$ is the same as the number of possible wedge products which map onto $\Q(\zeta_{2^k}-\zeta_{2^k}^{-1})$.
\end{proof}

\begin{Thm} The number of Schur rings over $Z_{2^n}$, where $n \ge 3$, is given by the following  equation:
\begin{equation}\label{eq:2formula1st} \Omega(n) = \Omega(n,0) + \Omega(n,2) + \sum_{k=3}^n (\Omega(n,k) + 2\Omega_\s(n,k)).
\end{equation}
\end{Thm} 
\begin{proof}
There is exactly one field in the $0$th layer and the second layer of $\L_{2^n}$. Each other layer of $\L_{2^n}$ contains three fields: $\Q(\zeta_{2^n})$, $\Q(\zeta_{2^n}+\zeta_{2^n}^{-1})$, and $\Q(\zeta_{2^n}-\zeta_{2^n}^{-1})$. The equation then follows from \propref{prop:2formulaSameLayer}.
\end{proof}

A direct consequence of \eqref{eq:2formula1st} and \lemref{lem:formulaCatalan} is the following:
\begin{equation}\label{eq:2formula2nd} \Omega(n) = \Omega(n,0) +  \sum_{k=0}^{n-2} c_k\Omega(n-k,2) + 2\sum_{k=3}^n \Omega_\s(n,k) \end{equation}
Therefore, we seek to express $2\sum_{k=3}^n \Omega_\s(n,k)$ in terms of the $\Omega(n,k)$.

\begin{Prop} For $k > 3$,
\begin{equation}\label{eq:SformulaMiddle} \Omega_\s(n,k) = \Omega_\s(n-1,k-1) + 2\sum_{j=k}^{n-1} \Omega_\s(n-1,j)\end{equation}
\end{Prop}
\begin{proof}
Following the same reasoning as \eqref{eq:formulaLowerLayers}, we see \eqref{eq:SformulaMiddle} is true for $k = n$ by \eqref{eq:SformulaTopLayer} and for $k < n$, it suffices to count the number of Schur rings $T$ over $Z_{2^{n-1}}$ for which $T\cap \Q[Z_{2^{k-1}}] = \s(Z_{2^{k-1}})$. This is exactly the number of Schur rings over $Z_{2^{n-1}}$ which map onto $\Q(\zeta_{2^j}+\zeta_{2^j}^{-1})$ for $k-1\le j \le n-1$ or onto $\Q(\zeta_{2^j}-\zeta_{2^j}^{-1})$ for $k-1 <j \le n-1$. The result then follows from \propref{prop:2formulaSameLayer}.
\end{proof}

\begin{Prop} For $n > 3$,
\begin{equation}\label{eq:SformulaBase} \Omega_\s(n,3) = \Omega(n-1,2) + 2\sum_{j=3}^{n-1} \Omega_\s(n-1,j)\end{equation}
\end{Prop}
\begin{proof}
Following the same reasoning as \eqref{eq:formulaLowerLayers}, it suffices to count the number of Schur rings $T$ over $Z_{2^{n-1}}$ for which $T \cap \Q[Z_{4}] = \Q[Z_2]\wr \Q[Z_2]$. This includes the Schur rings over $Z_{2^{n-1}}$ which map onto $\Q(\zeta_{2^j}+\zeta_{2^j}^{-1})$ for $3\le j \le n-1$ or onto $\Q(\zeta_{2^j}-\zeta_{2^j}^{-1})$ for $3 \le j \le n-1$. On the other hand, no Schur ring which maps onto $\Q(\zeta_{2^j})$ has this property for $j >1$. It remains to examine which Schur rings that map onto $\Q$ have this property. By \thmref{Thm:Claim7}, any Schur ring over $Z_{2^{n-1}}$ mapping onto $\Q$ has the form $T = \Q[Z_2]\wr T'$ for some Schur ring $T'$ over $Z_{2^{n-2}}$ such that $T' \cap \Q[Z_2] = \Q[Z_2]$, since $T\cap \Q[Z_4] = \Q[Z_2]\wr\Q[Z_2]$. As was seen in the proof of \propref{prop:formulaSecondLayer}, the number of choices for $T'$ is $\Omega(n-1,2)$.  The result then follows from \propref{prop:2formulaSameLayer}.
\end{proof}

Next, we need to expand $2\sum_{k=3}^n \Omega_\s(n,k)$ using \eqref{eq:SformulaMiddle} and \eqref{eq:SformulaBase}. This will produce an equation of the following form:
\begin{equation}\label{eq:SchroderCoefficients} 2\sum_{k=3}^n \Omega_\s(n,k) = \sum_{i=1}^{n-2} s_{i}\Omega(n-i, 2)\end{equation} for some positive integers $s_i$. In particular, the $j$th iteration of \eqref{eq:SformulaMiddle} and \eqref{eq:SformulaBase} will produce an equation of the form 
\begin{equation}\label{eq:SuperCatalanCoefficientsTriangle} 2\sum_{k=3}^n \Omega_\s(n,k) = \sum_{i=1}^{j-1} s_{i}\Omega(n-i,2) + \sum_{k=j+1}^{n} s_{jk} \Omega_\s(n-j, k-j)\end{equation} for some positive integers $s_{jk}$. We note that $s_{i(i+1)} = s_i$ and $s_{0k} = 1$ for all $k$. Furthermore, 
\begin{equation}\label{eq:SuperCatalanTriangle1} s_{jk} = s_{(j-1)k} + 2\sum_{\ell=j}^{k-1} s_{(j-1)\ell}\end{equation} by \eqref{eq:SformulaMiddle}. 
When $0 < j < k-1$, \eqref{eq:SuperCatalanTriangle1} can be rewritten recursively to give 
\begin{equation}\label{eq:SuperCatalanTriangle2} s_{jk} = s_{(j-1)k} + s_{j(k-1)} + s_{(j-1)(k-1)}.\end{equation}  From \eqref{eq:SuperCatalanTriangle2}, we can create a triangular array of integers, depicted in \figref{tab:SuperCatalanTriangleSmall}, where $k$ indexes the rows ($k \ge 1$) and $j$ indexes the columns ($0\le j < k$). The diagonal entries of the triangle give the values of $s_i$.

\begin{figure}[htbp]
\begin{center}
\caption{The Triangular Array of $s_{jk}$ Coefficients }
\begin{tabular}{cccccccc}
1 & &  &  &  &  &  &  \\ 
1 & 3 & &  &  &  &  &  \\ 
1 & 5 & 11 &  &  &  &  &  \\ 
1 & 7 & 23 & 45 &  &  &  &  \\ 
1 & 9 & 39 & 107 & 197 &  &  &  \\ 
1 & 11 & 59 & 205 & 509 & 903 &  &  \\ 
1 & 13 & 83 & 347 & 1061 & 2473 & 4279 & \\ 
1 & 15 & 111 & 541 & 1949 & 5483 & 12235 & 20793 \\ 
\end{tabular}
\label{tab:SuperCatalanTriangleSmall}
\end{center}
\end{figure}

\begin{Lem}\label{lem:formulaSchroder} Let $s_i$ be the coefficients given in Equation \eqref{eq:SchroderCoefficients}. Then $s_i = \sum_{j=0}^i \frac{1}{j+1}\binom{2j}{2}\binom{i+j}{2j}$, that is, $s_i$ is the $i$th Schr\"{o}der number.
\end{Lem}
\begin{proof}
Like in \lemref{lem:formulaSchroder}, we define $s_{00} =1$ and $s_{jj} = s_{(j-1)j}$ for $j > 0$. Now, this new triangular array is known as the Super-Catalan Triangle. One property of this triangle is that the sequence of diagonal entries is the sequence of super-Catalan numbers, also known as the little Schr\"{o}der numbers \cite{SuperCatalanTriangle}. Multiplying the little Schr\"{o}der numbers by two and reindexing gives the Schr\"oder numbers. 
\end{proof}


\begin{Thm}\label{thm:2formula3rd} The number of Schur rings over $Z_{2^n}$, where $n \ge 2$, is given by the following recursive equation:
\begin{equation}\label{eq:2formula3rd} \Omega(n) = \sum_{k=1}^3 2^k\Omega(n-k) - (c_{n-1} + s_{n-1}) + \sum_{k=4}^n\left(c_{k-1} + s_{k-1}-\sum_{j=1}^{k-3}(c_j+s_j)\right)\Omega(n-k) \end{equation}where $\Omega(0) =1$, $\Omega(1) = 1$, $\Omega(2) = 3$, $\Omega(3) = 10$, $c_k = \dfrac{1}{k+1}\dbinom{2k}{k}$ is the $k$th Catalan number, and $s_k = \sum_{j=0}^k \frac{1}{j+1}\binom{2j}{2}\binom{k+j}{2j}$  is the $k$th Schr\"{o}der number.
\end{Thm} 

For $n < 4$, we consider the second sum in \eqref{eq:2formula3rd} to be empty. Also, we define $\Omega(-1) = 0$, which appear in \eqref{eq:2formula3rd} for $n=2$.

\begin{proof}
By \eqref{eq:2formula2nd}, 
\[ \Omega(n) = \Omega(n,0) +  \sum_{k=0}^{n-2} c_k\Omega(n-k,2) + 2\sum_{k=3}^n \Omega_\s(n,k),\] which by \lemref{lem:formulaSchroder}, can be rewritten as 
\begin{eqnarray*}
 \Omega(n) &=& \Omega(n,0) +  \sum_{k=0}^{n-2} c_k\Omega(n-k,2) + \sum_{k=1}^{n-2} s_k\Omega(n-k,2)\\
 &=& \Omega(n,0) +  \Omega(n,2) + \sum_{k=1}^{n-2} (c_k+s_k)\Omega(n-k,2)\\
 &=& \Omega(n,0) +  \Omega(n,2) + (c_{n-2}+s_{n-2}) +  \sum_{k=1}^{n-3} (c_k+s_k)\Omega(n-k,2).
\end{eqnarray*}
We next can apply \propref{prop:formulaSecondLayer} to the above equation:
\begin{eqnarray*}
\Omega(n) &=& \Omega(n,0) +  [\Omega(n-1) - \Omega(n-2,0)] + (c_{n-2}+s_{n-2})\\
&&\qquad\qquad +  \sum_{k=1}^{n-3} (c_k+s_k)[\Omega(n-k-1) - \Omega(n-k-2,0)]\\
&=& \Omega(n-1) + \sum_{k=1}^{n-3}(c_{k}+s_{k})\Omega(n-k-1) + (c_{n-2}+s_{n-2}) + \Omega(n,0)\\
&& \qquad\qquad - \Omega(n-2,0) - \sum_{k=1}^{n-3} (c_k + s_k)\Omega(n-k-2,0)\\
&=& \Omega(n-1) + \sum_{k=2}^{n-2}(c_{k-1}+s_{k-1})\Omega(n-k) + (c_{n-2}+s_{n-2}) + \Omega(n,0)\\
&& \qquad\qquad - \Omega(n-2,0) - \sum_{k=3}^{n-1} (c_{k-2} + s_{k-2})\Omega(n-k,0)\\
&=& \Omega(n-1) + \sum_{k=2}^{n-1}(c_{k-1}+s_{k-1})\Omega(n-k) + \Omega(n,0)- \Omega(n-2,0)\\
&& \qquad\qquad  - \sum_{k=3}^{n-1} (c_{k-2} + s_{k-2})\Omega(n-k,0).
\end{eqnarray*}
Next we apply \propref{prop:formulaZeroLayer} to the above equation:
\[
\Omega(n) =  2\Omega(n-1) + \Omega(n-2) + \sum_{k=2}^{n-1}(c_{k-1}+s_{k-1})\Omega(n-k) - \sum_{k=3}^{n-1} (c_{k-2} + s_{k-2})\sum_{j=0}^{n-k-1}\Omega(j).\]
We note that 
\begin{eqnarray*}
\sum_{k=3}^{n-1} (c_{k-2} + s_{k-2})\sum_{j=0}^{n-k-1}\Omega(j) &=& \sum_{k=3}^{n-1}\sum_{j=0}^{n-k-1} (c_{k-2} + s_{k-2})\Omega(j) = \sum_{j=0}^{n-4}\sum_{k=3}^{n-j-1} (c_{k-2} + s_{k-2})\Omega(j)\\
&=& \sum_{k=0}^{n-4}\sum_{j=3}^{n-k-1} (c_{j-2} + s_{j-2})\Omega(k) =  \sum_{k=4}^{n}\sum_{j=3}^{k-1} (c_{j-2} + s_{j-2})\Omega(n-k).\\
\end{eqnarray*}
Therefore,
\begin{eqnarray*}
\Omega(n) &=& 2\Omega(n-1) + \Omega(n-2) + \sum_{k=2}^{n-1}(c_{k-1}+s_{k-1})\Omega(n-k) - \sum_{k=4}^{n}\sum_{j=3}^{k-1} (c_{j-2} + s_{j-2})\Omega(n-k)\\
&=& 2\Omega(n-1) + 4\Omega(n-2) + 8\Omega(n-3) - (c_{n-1} + s_{n-1})\\
&& \qquad\qquad + \sum_{k=2}^{n-1}\left(c_{k-1}+s_{k-1} - \sum_{j=3}^{k-1}(c_{j-2}+s_{j-2})\right)\Omega(n-k) \\
&=& \sum_{k=1}^3 2^k\Omega(n-k) - (c_{n-1} + s_{n-1}) + \sum_{k=4}^n\left(c_{k-1} + s_{k-1}-\sum_{j=1}^{k-3}(c_j+s_j)\right)\Omega(n-k). \qedhere
\end{eqnarray*}
\end{proof}

 \tableref{table:evenNumberSRings} lists the number of Schur rings over $Z_{2^n}$ up to the tenth power.

\begin{table}[hbtp]
\caption{Number of Schur Rings over $Z_{2^n}$}
\begin{center}
\begin{tabular}{|c|c|c|c|c|c|c|c|c|c|c|}
\hline
$n$ & \textbf{1} & \textbf{2} & \textbf{3} & \textbf{4} & \textbf{5} & \textbf{6} & \textbf{7} & \textbf{8} & \textbf{9} & \textbf{10} \\ \hline
$\Omega(n)$ & 1 & 3 & 10 & 37 & 151 & 657 & 2,989 & 14,044 & 67,626 & 332,061 \\ \hline
\end{tabular}
\end{center}
\label{table:evenNumberSRings}
\end{table}

Let $\mathcal F(z) = \sum_{n=0}^\infty \Omega(n)z^n$ be the generating function of $\Omega$, for $p=2$. Using \[\mathcal C(z) = \sum_{n=0}^\infty c_nz^n = \dfrac{1-\sqrt{1-4z}}{2z},\] the generating function for the Catalan numbers, \[\mathcal S(z) = \sum_{n=0}^\infty s_nz^n = \dfrac{1-z-\sqrt{1-6z+z^2}}{2z},\] be the generating function for the Schr\"oder numbers, and the usual generating function calculations, one computes that 
\begin{eqnarray} 
\mathcal F(z)  
&=& \dfrac{(2-z-\sqrt{1-4z}-\sqrt{1-6z+z^2})(1-z) + 2(z^2-1)}{(2-z-\sqrt{1-4z}-\sqrt{1-6z+z^2})(1-z-z^2) + 2(z^3+z^2+z-1)}.
\end{eqnarray}

\underline{Acknowledgments}: The contents of this paper are part of the author's doctoral dissertation \cite{MePhD}, written under the supervision of Stephen P. Humphries, that was submitted to Brigham Young University. All computations made in preparation of this paper were accomplished using the computer software Magma \cite{Magma} and the exact code can be found in \cite{MePhD} and \cite{Kerby}, the second reference being the master's thesis of Brent Kerby, another student of Humphries. Helpful suggestions, instructions, and guidance were offered to the author by Michael Barrus on many of combinatorical topics in this paper. Finally, acknowledgment should be given by Petr Vojtechovsky for his insights on how the topics of this paper related to the PORC conjecture.

\bibliographystyle{plain}
\bibliography{Srings}

\begin{thebibliography}{10}

\bibitem{Magma}
Wieb Bosma, John Cannon, and Catherine Playoust.
\newblock The {M}agma algebra system. {I}. {T}he user language.
\newblock {\em J. Symbolic Comput.}, 24(3-4):235--265, 1997.
\newblock Computational algebra and number theory (London, 1993).

\bibitem{VaughanLeeIII}
Marcus du~Sautoy and M.R. Vaughan-Lee.
\newblock Non-porc behaviour of a class of descendant $p$-groups.
\newblock {\em Journal of Algebra}, 361:287--312, 2012.

\bibitem{evseev2008higman}
Anton Evseev.
\newblock Higman's porc conjecture for a family of groups.
\newblock {\em Bulletin of the London Mathematical Society}, 40(3):415--431,
  2008.

\bibitem{SuperCatalanTriangle}
Johannes Fischer.
\newblock Sequence a144944.
\newblock {\em The On-Line Encyclopedia of Integer Sequences}.

\bibitem{HigmanPORCI}
G.~Higman.
\newblock Enumerating $p$-groups. i: Inequalities.
\newblock {\em Proceedings of the London Mathematical Society}, 10(3):24--30,
  1960.

\bibitem{HigmanPORCII}
G.~Higman.
\newblock Enumerating $p$-groups. ii: Problems whose solution is porc.
\newblock {\em Proceedings of the London Mathematical Society}, 10(3):566--582,
  1960.

\bibitem{Kerby}
Brent Kerby.
\newblock Rational schur rings over abelian groups.
\newblock Master's thesis, Brigham Young University, 2008.

\bibitem{KlinPoschel}
M.~Kh. Klin and R.~P{o}schel.
\newblock The k{o}nig problem, the isomorphism problem for cyclic graphs and
  the method of schur rings.
\newblock {\em Algebraic Methods in Graph Theory}, 1, 2, 1978.

\bibitem{Kovacs}
Istv\'{a}n Kov\'{a}cs.
\newblock The number of indecomposable schur rings over a cyclic 2-group.
\newblock {\em S\'{e}minaire Lotharingien de Combinatoire}, 51:Article B51h,
  2005.

\bibitem{Leung90}
Ka~Hin Leung and Siu~Lun Ma.
\newblock The structure of schur rings over cyclic groups.
\newblock {\em Journal of Pure and Applied Algebra}, 66:287--302, 1990.

\bibitem{LeungII}
Ka~Hin Leung and Shin~Hing Man.
\newblock On schur rings over cyclic groups ii.
\newblock {\em Journal of Algebra}, 183:273--285, 1996.

\bibitem{LeungI}
Ka~Hin Leung and Shin~Hing Man.
\newblock On schur rings over cyclic groups.
\newblock {\em Israel Journal of Mathematics}, 106:251--267, 1998.

\bibitem{Liskovets}
V.~Liskovets and R.~P{o}schel.
\newblock Counting circulant graphs of prime-power order by decomposing into
  orbit enumeration problems.
\newblock {\em Discr. Math.}, 214:173--191, 2000.

\bibitem{Ma}
S.~L. Ma.
\newblock On association schemes, schur rings, strongly regular graphs and
  partial difference sets.
\newblock {\em Ars Combin.}, 21:211--220, 1989.

\bibitem{MePhD}
Andrew Misseldine.
\newblock {\em Algebraic and Combinatorial Properties of Schur Rings over
  Cyclic Groups}.
\newblock PhD thesis, Brigham Young University, 2014.

\bibitem{Muzychuk98}
Mikhail Muzychuk.
\newblock The structure of schur rings over cyclic groups of square-free order.
\newblock {\em Acta Applicandae Mathematicae}, 52:163--181, 1998.

\bibitem{Muzychuk93}
Mikhail~E. Muzychuk.
\newblock The structure of rational schur rings over cyclic groups.
\newblock {\em European Journal of Combinatorics}, 14:479--490, 1993.

\bibitem{Muzychuk94}
Mikhail~E. Muzychuk.
\newblock On the structure of basic sets of schur rings over cyclic groups.
\newblock {\em Journal of Algebra}, 169:655--678, 1994.

\bibitem{VaughanLeeI}
M.F. Newman, E.A. O'Brien, and M.R. Vaughan-Lee.
\newblock Groups and nilpotent lie rings whose order is the sixth power of a
  prime.
\newblock {\em Journal of Algebra}, 278(1):383--401, 2004.

\bibitem{VaughanLeeII}
E.A. O'Brien and M.R. Vaughan-Lee.
\newblock The groups with order $p^7$ for odd prime $p$.
\newblock {\em Journal of Algebra}, 292(1):243--258, 2005.

\bibitem{CatalanTriangle}
N.~J.~A. Sloane.
\newblock Sequence a009766.
\newblock {\em The On-Line Encyclopedia of Integer Sequences}.

\bibitem{Wielandt49}
Helmut Wielandt.
\newblock Zur theorie der einfach transitiven permutationsgruppen {II
  (German)}.
\newblock {\em Math. Z.}, 52:384--393, 1949.

\bibitem{Wielandt64}
Helmut Wielandt.
\newblock {\em Finite Permutation Groups}.
\newblock Academic Press, New York-London, 1964.

\bibitem{WittyPORC}
Brett~E. Witty.
\newblock Enumeration of groups of prime-power order.
\newblock {\em Bulletin of the Australian Mathematical Society}, 76:479--480,
  12 2007.

\end{thebibliography}
\end{document}